\pdfoutput=1
\RequirePackage[l2tabu, orthodox]{nag}

\documentclass[reqno]{amsart}

\usepackage{lmodern}
\usepackage[T1]{fontenc}
\usepackage[utf8]{inputenc}
\usepackage[english]{babel}
\usepackage{microtype} 

\usepackage{amsmath,amssymb,amsthm,mathrsfs,latexsym,mathtools,mathdots,enumerate,tikz,bm,url,comment}
\usepackage[centertableaux]{ytableau}

\usepackage{hyperref}

\usepackage[capitalize,noabbrev]{cleveref}

\newtheorem{theorem}{Theorem}
\newtheorem{proposition}[theorem]{Proposition}
\newtheorem{lemma}[theorem]{Lemma}
\newtheorem{corollary}[theorem]{Corollary}
\newtheorem{conjecture}[theorem]{Conjecture}

\theoremstyle{definition}
\newtheorem{definition}[theorem]{Definition}
\newtheorem{example}[theorem]{Example}
\newtheorem{problem}[theorem]{Problem}

\newcommand{\defin}[1]{\emph{#1}}

\newcommand{\opsurj}{\mathcal{O}}

\newcommand{\setN}{\mathbb{N}}
\newcommand{\setZ}{\mathbb{Z}}

\newcommand{\avec}{\mathbf{a}}
\newcommand{\evec}{\mathbf{e}}
\newcommand{\xvec}{\mathbf{x}}
\newcommand{\svec}{\mathbf{s}}

\newcommand{\hrvp}{{\pi}} 
\newcommand{\hrvpp}{{\overline{\pi}}} 

\newcommand{\psumP}{\mathrm{p}}
\newcommand{\elementaryE}{\mathrm{e}}
\newcommand{\completeH}{\mathrm{h}}
\newcommand{\macdonaldH}{\mathrm{\tilde H}}
\newcommand{\hallLittlewoodT}{\mathrm{H}}
\newcommand{\schurS}{\mathrm{s}}
\newcommand{\LLT}{\mathrm{G}}
\newcommand{\FLLT}{\tilde{\mathrm{G}}}
\newcommand{\LLTc}{\hat{\mathrm{G}}} 
\newcommand{\chrom}{\mathrm{X}}

\newcommand{\SSYT}{\mathrm{SSYT}}
\newcommand{\SYT}{\mathrm{SYT}}

\newcommand{\Des}{\mathrm{Des}}

\DeclareMathOperator{\wt}{wt}
\DeclareMathOperator{\modwt}{\widetilde{wt}}

\DeclareMathOperator{\std}{std}

\DeclareMathOperator{\hrv}{hrv} 
\DeclareMathOperator{\length}{\ell}

\DeclareMathOperator{\asc}{asc}

\title{LLT polynomials, elementary symmetric functions and melting lollipops}

\author{Per Alexandersson}
\address{Dept. of Mathematics, Royal Institute of Technology, SE-100 44 Stockholm, Sweden}
\email{per.w.alexandersson@gmail.com}

\begin{document}

\begin{abstract}
We conjecture an explicit positive combinatorial formula for 
the expansion of unicellular LLT polynomials in the elementary symmetric basis.
This is an analogue of the Shareshian--Wachs conjecture previously
studied by Panova and the author in 2018.
We show that the conjecture for unicellular LLT polynomials 
implies a similar formula for vertical-strip LLT polynomials.

We prove positivity in the elementary basis in for the class of graphs 
called ``melting lollipops'' previously considered by Huh, Nam and Yoo.
This is done by proving a curious relationship between a generalization of charge 
and orientations of unit-interval graphs. 

We also provide short bijective proofs of Lee's three-term recurrences
for unicellular LLT polynomials and we show that these recurrences are enough 
to generate all unicellular LLT polynomials associated with abelian area sequences.
\end{abstract}

\maketitle

\setcounter{tocdepth}{1}
\tableofcontents

\section{Introduction}

\subsection{Background on LLT polynomials}

LLT polynomials were introduced by Lascoux, Leclerc and Thibon in \cite{Lascoux97ribbontableaux},
and are $q$-deformations of products of skew Schur functions.
An alternative combinatorial model for the LLT polynomials was later introduced in \cite{HaglundHaimanLoehr2005}
while studying Macdonald polynomials. In their paper, LLT polynomials are indexed by a $k$-tuple of skew shapes.
In the case each such skew shape is a single box, the LLT polynomial is said to be \emph{unicellular LLT polynomial}.
Such unicellular LLT polynomials are the main topic of this paper.

\subsection{Background on chromatic symmetric functions}

In \cite{CarlssonMellit2017} Carlsson and Mellit introduced a more convenient combinatorial model 
for unicellular LLT polynomials, indexed by (area sequences of) Dyck paths.  
They also highlighted an important relationship using plethysm between unicellular LLT polynomials
and the chromatic quasisymmetric functions introduced by Shareshian and Wachs in \cite{ShareshianWachs2012}.

The chromatic quasisymmetric functions refine the chromatic symmetric 
functions introduced by Stanley in \cite{Stanley95Chromatic}.
The Stanley--Stembridge conjecture \cite{StanleyStembridge1993} states 
that such chromatic symmetric functions 
associated with \defin{unit interval graphs}, and more generally, 
incomparability graphs of $3+1$-free posets are positive in the elementary symmetric basis, or $\elementaryE$-positive for short.
Their conjecture was refined with the introduction of an additional parameter $q$
in \cite{ShareshianWachs2012}. 
The class of graphs for which this conjecture is believed to hold 
was later extended to the class of \emph{circular unit interval graphs} in \cite{Ellzey2016,Ellzey2017}
where it is conjectured that the chromatic quasisymmetric functions expanded in the $\elementaryE$-basis
have coefficients in $\setN[q]$, see \cref{conj:shareshianWachs} below.
To this date, there is still not even a conjectured 
combinatorial formula for the $\elementaryE$-expansion of the chromatic symmetric functions.

The idea of studying LLT polynomials in parallel with quasisymmetric chromatic functions
originated in \cite{CarlssonMellit2017}, althought the connection is perhaps in hindsight apparent 
in the techniques used in \cite{HaglundHaimanLoehr2005}.
We also mention an interesting paper by Haglund 
and Wilson \cite{HaglundWilson2017} explores the connection between
the integral-form Macdonald polynomials and the quasisymmetric chromatic functions.

\subsection{Main results}

In \cite{AlexanderssonPanova2016}, we stated an analogue of the Shareshian--Wachs conjecture
regarding $\elementaryE$-positivity of unicellular LLT polynomials, $\LLT_\avec(\xvec;q+1)$
and proved the conjecture in a few cases.
We also provided many similarities between unicellular LLT polynomials and 
chromatic quasisymmetric functions associated with unit-interval graphs.
The problem of $\elementaryE$-positivity of 
unicellular LLT polynomials is the main topic of this article.

\medskip 

The main results are:
\begin{itemize}
 \item We present a precise conjectured combinatorial formula for the $\elementaryE$-expansion of 
 $\LLT_\avec(\xvec;q+1)$. Our conjecture states that the unicellular LLT polynomial $\LLT_{\avec}(\xvec;q)$ is given as
 \begin{align}\label{eq:conjFormulaIntro}
  \LLT_{\avec}(\xvec;q+1) \coloneqq \sum_{\theta \in O(\avec)} q^{\asc(\theta)}\elementaryE_{\hrvpp(\theta)}(\xvec).
\end{align}
where $O(\avec)$ is the set of orientations of the unit interval graph with area sequence $\avec$,
and $\hrvpp(\theta)$ is an explicit partition-valued statistic on such orientation.
This formula can be extended to vertical-strip LLT polynomials,
and has been verified on the computer for all unit-interval graphs up to $10$ vertices.
This formula is surprising, as there is still no analogous conjectured formula for chromatic symmetric functions.
 
 A possible application of \eqref{eq:conjFormulaIntro} is to find a positive combinatorial 
 formula for the Schur-expansion of $\LLT_{\avec}(\xvec;q)$.
 
 \item We prove in \cref{cor:unicellularImpliesVS} that the conjectured 
 formula \eqref{eq:conjFormulaIntro} implies a generalized formula 
 for the so called \emph{vertical-strip LLT polynomials}.
 Furthermore, we prove that \eqref{eq:conjFormulaIntro} holds for 
 the family of complete graphs and line graphs.

 \item 
 Analogous recursions for the unicellular LLT polynomials are given by Lee in \cite{Lee2018}.
 We give short bijective proofs of these recurrences and show that 
 all graphs associated with abelian Hessenberg varieties can 
 be computed recursively via Lee's recurrences, starting from unicellular LLT polynomials 
 associated with the complete graphs.
 
 \item In \cref{sec:hallLittlwood}, we prove that the transformed Hall--Littlewood polynomials 
 $\hallLittlewoodT_\lambda(\xvec;q+1)$ are positive in the complete homogeneous basis.
 This implies that a corresponding family of vertical-strip LLT polynomials are $\elementaryE$-positive.
 
  Note that vertical-strip LLT polynomials appear in diagonal harmonics,
 see for example \cite[Section 4]{Bergeron2017} and \cite{HaglundHaimanLoehrRemmelUlyanov2005,Bergeron2013}.
 Consequently, \eqref{eq:conjFormulaIntro} provides support for some of the conjectures 
 regarding $\elementaryE$-positive in these references.
 We note that the authors of a recent preprint \cite{GarsiaHaglundQiuRomero2019} also 
 independently found the conjecture in \eqref{eq:conjFormulaIntro}. 
 The $\elementaryE$-positivity part of the conjecture has since been proved by M. D'Adderio in \cite{DAdderio2019}.
 We remark that $\elementaryE$-positivity is very rare in reality, see \cite{PatriasWilligenburg2018} for details.
 
 \item In \cref{sec:cocharge}, we prove a curious identity between a generalization of charge, 
 denoted $\wt_\avec(T)$, and the set of orientations, $O(\avec)$, of a unit-interval graph $\Gamma_\avec$.
 It states that
 \begin{align}\label{eq:chargeFormulaIntro}
  \sum_{\lambda \vdash n} \sum_{T \in \SYT(\lambda)} (q+1)^{\wt_\avec(T)} \schurS_\lambda(\xvec) = 
  \sum_{\theta \in O(\avec)} q^{\asc(\theta)} \elementaryE_{\sigma(\theta)}(\xvec),
 \end{align}
 where $\asc(\cdot)$ and $\sigma(\cdot)$ are certain combinatorial statistics on orientations.
 This version of charge was considered in \cite{HuhNamYoo2018} in order to prove Schur positivity for unicellular 
 LLT polynomials in the melting lollipop graph case.
 
 As a consequence, we get an explicit positive $\elementaryE$-expansion the case of \emph{melting lollipop graphs}
 which has previously been considered in \cite{HuhNamYoo2018}.
 The corresponding family of chromatic quasisymmetric functions was 
 considered in \cite{DahlbergWilligenburg2018} where they were proved to be $\elementaryE$-positive.
 Note however that the statistic $\hrvpp(\theta)$ in \eqref{eq:conjFormulaIntro} and $\sigma(\theta)$ in 
 \eqref{eq:chargeFormulaIntro} are different.
\end{itemize}

The paper is organized as follows. 
We first introduce the family of unicellular-- and vertical-strip LLT polynomials
and some of their basic properties. In \cref{sec:recProp}, we prove 
several recursive identities for such LLT polynomials.
In particular, we show that the recursions by Lee \cite{Lee2018}
can be used to construct unicellular LLT polynomials indexed by any abelian area sequence.

Some vertical-strip LLT polynomials are closely related to 
the transformed Hall--Littlewood polynomials. 
In \cref{sec:hallLittlwood}, we show that the transformed Hall--Littlewood polynomials 
$\hallLittlewoodT_\lambda(\xvec;q+1)$ are $\completeH$-positive, which gives further 
support for the main conjecture.

In \cref{sec:cocharge}, we study the relationship between a type of generalized cocharge
introduced in \cite{HuhNamYoo2018} and $\elementaryE$-positivity. 
This provides a proof that unicellular LLT polynomials given by melting lollipop graphs 
are $\elementaryE$-positive. 

Finally in \cref{sec:powersum}, we describe a possible approach 
to prove \eqref{eq:conjFormulaIntro} by a comparison in the power-sum symmetric basis.

\section{Preliminaries}

We use the same notation and terminology as in \cite{AlexanderssonPanova2016}.
The reader is assumed to have a basic background on symmetric functions and related combinatorial objects, 
see \cite{StanleyEC2,Macdonald1995}. 
All Young diagrams and tableaux are presented in the English convention.

\subsection{Dyck paths and unit-interval graphs}\label{subsec:graphs}

An \defin{area sequence} is an integer vector $\avec=(a_1,\dotsc,a_n)$ which satisfies
\begin{itemize}
 \item $0 \leq a_i \leq i-1 $ for $1 \leq i \leq n$ and
 \item $a_{i+1} \leq a_{i} + 1$ for $1 \leq i < n$.
\end{itemize}
The number of such area sequences of size $n$ is given by the Catalan numbers.
Note that \cite{HuhNamYoo2018} uses a reversed indexing of entries in area sequences.

\begin{definition}\label{def:UIGraph}
For every area sequence of length $n$, we define a \defin{unit interval graph} 
$\Gamma_\avec$ with vertex set $[n]$ and the directed edges
\begin{align}\label{eq:cyclicUIGraph}
(i-a_i) \to i, \; (i-a_i + 1) \to i,  \; (i-a_i + 2) \to i, \; \dotsc \; , (i-1) \to i
\end{align}
for all $i = 1,\dotsc,n$.
We say that $(u,v)$ with $u<v$ is an \defin{outer corner} of $\Gamma_\avec$
if $(u,v)$ is not an edge of $\Gamma_\avec$, and either
\begin{itemize}
 \item $u+1=v$ or
 \item $(u+1,v)$ and $(u,v-1)$ are edges of $\Gamma_\avec$.
\end{itemize}
\end{definition}

\begin{example}\label{ex:dyckPath}
We can illustrate area sequences and their corresponding unit-interval graphs as \defin{Dyck diagrams}, 
as is done in \cite{qtCatalanBook,AlexanderssonPanova2016}.
For example, $(0,1,2,3,2,2)$ corresponds to the diagram
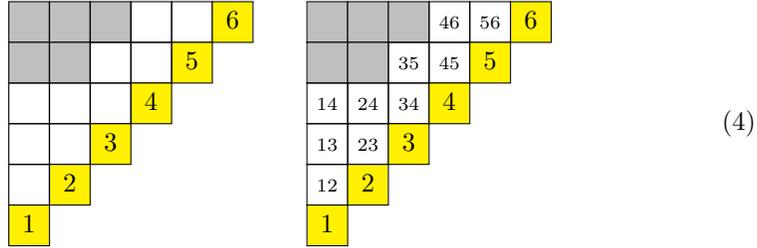
\begin{align}\label{eq:dyckPathExample}
\begin{ytableau}
*(lightgray)   & *(lightgray)   & *(lightgray)   &      &    & *(yellow) 6 \\
*(lightgray)   & *(lightgray)   &     &     & *(yellow) 5\\
    &    &     & *(yellow) 4\\
    &     & *(yellow) 3\\
    & *(yellow) 2\\
*(yellow) 1
\end{ytableau}
\qquad 
\begin{ytableau}
*(lightgray)   & *(lightgray)   & *(lightgray)   &   \scriptstyle{46}   & \scriptstyle{56}   & *(yellow) 6 \\
*(lightgray)   & *(lightgray)   &  \scriptstyle{35}   &   \scriptstyle{45}  & *(yellow) 5\\
   \scriptstyle{14} & \scriptstyle{24}   &   \scriptstyle{34}  & *(yellow) 4\\
  \scriptstyle{13}  &  \scriptstyle{23}   & *(yellow) 3\\
\scriptstyle{12}    & *(yellow) 2\\
*(yellow) 1
\end{ytableau}
\end{align}
where the area sequence specify the number of white squares in each row, bottom to top.
The squares on the main diagonal are the vertices of $\Gamma_\avec$,
and each white square correspond to a directed edge of $\Gamma_\avec$.
In the second figure we see this correspondence where edge $(i,j)$ is marked as $ij$.
The outer corners of $\Gamma_\avec$ are $(2,5)$ and $(3,6)$.
\end{example}

\textbf{Caution:} We do not really distinguish the 
terms \emph{area sequence}, \emph{Dyck diagram} and \emph{unit interval graph},
as they all relate to the same objects. 
What term is used depends on context and what features we wish to emphasize.
\medskip

Let $\Gamma_\avec$ be an unit interval graph with $n$ vertices. 
We let $\avec^T$ denote the area sequence of $\Gamma_\avec$ where all edges have 
been reversed, and every vertex $j \in [n]$ has been relabeled with $n+1-j$.
This operation corresponds to simply transposing the Dyck diagram.
\begin{lemma}[{See \cite{AlexanderssonPanova2016}}]\label{lem:rowColArea}
The entries in an area sequence $\avec$ is a rearrangement of the entries in $\avec^T$.
\end{lemma}

Most results in this paper concerns a few special classes of area sequences.
\begin{definition}
An area sequence of length $n$ is called \defin{rectangular} if 
either $\avec = (0,1,2,\dotsc,n-1)$ or there is some $k \in [n]$
such that
\[
 a_i = i-1 \text{ for $i=1,2,\dotsc,k$ and } a_j = j-k-1 \text{ for $j=k+1,k+2,\dotsc,n$}.
\]
This condition is equivalent with all non-edges forming a $k \times (n-k)$-rectangle in the Dyck diagram.
Furthermore, an area sequence $\avec'$ is called \defin{abelian} 
whenever $a'_i \geq a_i$ for some rectangular sequence $\avec$.
For example, the area sequence in \eqref{eq:dyckPathExample} is abelian.
\end{definition}
The terminology is motivated by \cite{HaraldaPrecup2017}, 
where abelian area sequences are associated with abelian Hessenberg varieties.

We will also consider the following families of area sequences:
\begin{itemize}
 \item The complete graphs, $(0,1,2,\dotsc,n-1)$.
 \item The line graphs $(0,1,1,\dotsc,1)$.
 \item Lollipop graphs, where 
 \[
  a_i = \begin{cases}
	i-1 \text{ for } i=1,\dotsc,m \\
	1\text{ for } i=m+1,\dotsc,m+n
\end{cases}
 \]
 for some $m,n\geq 1$.
\item Melting complete graph,
 \[
  a_i = \begin{cases}
        i-1\text{ for } i=1,2,\dotsc,n-1 \\
	n-k-1 \text{ for } i=n 
\end{cases}
\]
where $0\leq k \leq n-1$.
\item Melting lollipop graphs, defined as 
 \[
  a_i = \begin{cases}
	i-1 \text{ for } i=1,\dotsc,m-1 \\
	m-1-k \text{ for } i=m \\
	1\text{ for } i=m+1,\dotsc,m+n
\end{cases}
 \]
for $m,n\geq 1$ and $0\leq k \leq m-1$.
\end{itemize}

\subsection{Vertical strip diagrams}

A \defin{vertical strip diagram} is a Dyck diagram where some of the outer corners
have been marked with $\rightarrow$. We call such an outer corner a \defin{strict edge}.
These markings correspond to some extra oriented edges of $\Gamma_\avec$.
We use the notation $\Gamma_{\avec,\svec}$ to denote a directed graph 
with some additional strict edges $\svec$ and refer to the graph $\Gamma_{\avec,\svec}$
as a vertical strip diagram as well.
\begin{example}
Below is an example of a vertical strip diagram.
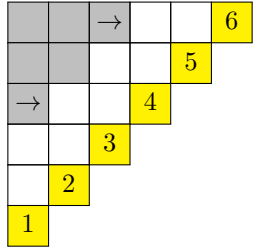
\[
\begin{ytableau}
*(lightgray)   & *(lightgray)   & *(lightgray) \rightarrow   &      &    & *(yellow) 6 \\
*(lightgray)   & *(lightgray)   &     &     & *(yellow) 5\\
*(lightgray)   \rightarrow &    &     & *(yellow) 4\\
    &     & *(yellow) 3\\
    & *(yellow) 2\\
*(yellow) 1
\end{ytableau}
\]
The edges $(1,4)$ and $(3,6)$ are strict of $\Gamma_{\avec,\svec}$, 
and the directed edges of $\Gamma_{\avec}$ (which are also edges of $\Gamma_{\avec,\svec}$)
are 
\[
 \{ (1,2), (1,3), (2,3), (2,4),(3,4), (3,5),(4,5), (4,6), (5,6) \}.
\]
Note that this is another example of a diagram with an abelian area sequence.
\end{example}

\subsection{Vertical strip LLT polynomials}

Let $\Gamma_{\avec,\svec}$ be a vertical strip diagram.
A \defin{valid coloring} $\kappa :V(\Gamma_{\avec,\svec}) \to \setN$ 
is a vertex coloring of $\Gamma_{\avec,\svec}$ 
such that $\kappa(u)<\kappa(v)$ whenever $(u,v)$ is a strict edge in $\svec$.
Given a coloring $\kappa$, an \emph{ascent} of $\kappa$ is a (directed) edge $(u,v)$ in $\Gamma_{\avec,\svec}$
such that $\kappa(u)<\kappa(v)$. Note that strict edges do not count as ascents.
Let $\asc(\kappa)$ denote the number of ascents of $\kappa$.

\begin{definition}\label{def:LLTPoly}
Let $\Gamma_{\avec,\svec}$ be a vertical strip diagram. 
The \emph{vertical strip LLT polynomial} $\LLT_{\avec,\svec}(\xvec;q)$ is defined as
\begin{equation}\label{eq:LLTdef}
 \LLT_{\avec,\svec}(\xvec;q) \coloneqq \sum_{\kappa : V(\Gamma_{\avec,\svec}) \to \setN } \xvec^\kappa q^{\asc(\kappa)}
\end{equation}
where the sum is over valid colorings of $\Gamma_{\avec,\svec}$.
Whenever $\svec = \emptyset$, we simply write $\LLT_{\avec}(\xvec;q)$ 
and refer to this as a \defin{unicellular LLT polynomial}.
\end{definition}
As an example, here is $\LLT_{0012}(\xvec;q)$ expanded in the Schur basis:
\[
  \LLT_{0012}(\xvec;q) = q^3 \schurS_{1111}+(q+q^2+q^3) \schurS_{211}+(q+q^2) \schurS_{22}+(1+q+q^2)\schurS_{31}+\schurS_{4}.
\]
The polynomials $\LLT_{\avec,\svec}(\xvec;q)$ are known to be symmetric,
and correspond to classical LLT polynomials indexed by $k$-tuples of skew shapes as in \cite{HaglundHaimanLoehr2005}.
In fact, the unicellular LLT polynomials correspond to the case when all shapes in the $k$-tuple are single cells, 
and the vertical strip case correspond to $k$-tuples of single columns.
This correspondence is proved in \cite{AlexanderssonPanova2016} and is also done implicitly in \cite{CarlssonMellit2017}.
There is a close connection with the $\zeta$ map used by Haglund and Loehr, see \cite{HaglundLoehr2005,qtCatalanBook}.

\begin{example}
In the following vertical strip diagram, we illustrate a valid coloring $\kappa$
where we have written $\kappa(i)$ on vertex $i$. 
That is, $\kappa(1)=1$, $\kappa(2)=3$, $\kappa(3)=2$, etc.
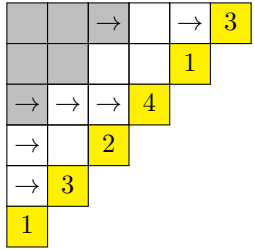
\[
\begin{ytableau}
*(lightgray)   & *(lightgray)   & *(lightgray) \rightarrow   &      & \to   & *(yellow) 3 \\
*(lightgray)   & *(lightgray)   &     &     & *(yellow) 1\\
*(lightgray)   \rightarrow &  \to  &  \to   & *(yellow) 4\\
  \to  &     & *(yellow) 2\\
  \to  & *(yellow) 3\\
*(yellow) 1
\end{ytableau}
\]
The strict edges and edges contributing to $\asc(\kappa)$ have been marked with $\to$. 
Hence, $\kappa$ contributes with $q^{5}x_1^2 x_2 x_3^2 x_4$ to the sum in \eqref{eq:LLTdef}.
\end{example}

\subsection{A conjectured formula}

\begin{definition}\label{def:orientationStuff}
Let $\avec$ be an area sequence of length $n$ and $\svec$ be some strict edges of $\Gamma_\avec$. 
Let $O(\avec,\svec)$ denote the set of orientations 
of the graph $\Gamma_\avec$ (seen as an undirected graph) together with the extra directed edges in $\svec$.
Thus, the cardinality of $O(\avec,\svec)$ is simply $2^{a_1+\dotsb+a_n}$.
If $\svec=\emptyset$, we simply write $O(\avec)$ for the set of orientations of $\Gamma_\avec$.
Given $\theta \in O(\avec,\svec)$, an edge $(u,v)$ is an \defin{ascending} edge in $\theta$
if it is oriented in the same manner as in $\Gamma_\avec$. 
Let $\asc(\theta)$ denote the number of ascending edges in $\theta$.
\emph{Note that edges in $\svec$ are not considered to be ascending!}

We now define the \defin{highest reachable vertex}, $\hrv_\theta(u)$ for $u\in [n]$ as
the maximal $v$ such that there is a path from $u$ to $v$ in $\theta$ \emph{using only strict and ascending edges}.
Note that $\hrv_\theta(u)\geq u$ for all $u$.
The orientation $\theta$ defines a \emph{set partition} $\hrvp(\theta)$ of the vertices of $\Gamma_\avec$,
where two vertices are in the same part if and only if they have the same highest reachable vertex.
Let $\hrvpp(\theta)$ denote the \emph{partition} given by the sizes of the sets in $\hrvp(\theta)$.

Let $\avec$ be an area sequence and $\svec$ be some strict edges of $\Gamma_\avec$.
Define the symmetric function $\LLTc_{\avec,\svec}(\xvec;q)$ via the relation
\begin{align}\label{eq:conjFormula}
\LLTc_{\avec,\svec}(\xvec;q+1) \coloneqq \sum_{\theta \in O(\avec,\svec)} q^{\asc(\theta)}\elementaryE_{\hrvpp(\theta)}(\xvec).
\end{align}
\end{definition}

\begin{example}
Below, we illustrate an orientation $\theta \in O(\avec,\svec)$,
where $\avec = (0,1,2,2,2,2)$ and $\svec = \{ (1,4), (2,5)  \}$.
As before, strict edges and edges contributing to $\asc(\theta)$ are marked with $\to$.
\[
\begin{ytableau}
*(lightgray)   & *(lightgray)   & *(lightgray)   &      & \to   & *(yellow) 6 \\
*(lightgray)   & *(lightgray)  \rightarrow &     &     & *(yellow) 5\\
*(lightgray)   \rightarrow &  \to &  \to   & *(yellow) 4\\
  \to  & \to    & *(yellow) 3\\
    & *(yellow) 2\\
*(yellow) 1
\end{ytableau}
\]
We have that $\hrv_\theta(2)=\hrv_\theta(5)=\hrv_\theta(6)=6$ and $\hrv_\theta(1)=\hrv_\theta(3)=\hrv_\theta(4)=4$.
Thus $\hrvp(\theta) = \{652,431\}$ and the orientation $\theta$ contributes 
with $q^{5}\elementaryE_{33}(\xvec)$ in \ref{eq:conjFormula}.
The full polynomial $\LLTc_{\avec,\svec}(\xvec;q+1)$ is
\begin{align*}
&(4 q^3+20 q^4+41 q^5+44 q^6+26 q^7+8 q^8+q^9) \elementaryE_{6}+
(2 q^2+7 q^3+9 q^4+5 q^5+q^6) \elementaryE_{33}+\\
&(2 q^2+9 q^3+16 q^4+14 q^5+6 q^6+q^7) \elementaryE_{42}+\\
&(4 q^2+22 q^3+48 q^4+53 q^5+31 q^6+9 q^7+q^8) \elementaryE_{51}+\\
&(4 q+14 q^2+18 q^3+10 q^4+2 q^5) \elementaryE_{321}+
(q+8 q^2+20 q^3+22 q^4+11 q^5+2 q^6) \elementaryE_{411}+\\
&(1+3 q+3 q^2+q^3) \elementaryE_{2211}+
(q+3 q^2+3 q^3+q^4) \elementaryE_{3111}
\end{align*}
\end{example}

\begin{conjecture}[Main conjecture]\label{conj:main}
For any vertical-strip LLT polynomial $\LLT_{\avec,\svec}(\xvec;q)$ 
we have that $\LLT_{\avec,\svec}(\xvec;q) = \LLTc_{\avec,\svec}(\xvec;q)$.
\end{conjecture}
Note that this conjecture implies that $\LLT_{\avec,\svec}(\xvec;q+1)$ is $\elementaryE$-positive,
with the expansion given as a sum over all orientations of $\Gamma_\avec$.
Such a conjecture was first stated in \cite{AlexanderssonPanova2016}
but without a precise definition of $\hrvpp(\theta)$.
\cref{conj:main} is a natural analogue of 
the Shareshian--Wachs conjecture, \cite{ShareshianWachs2012,ShareshianWachs2016},
and therefore is also closely related to the Stanley--Stembridge conjecture \cite{StanleyStembridge1993,Stanley95Chromatic}.
There is also a natural generalization of \cref{eq:conjFormula} that predicts the $\elementaryE$-expansion
of the LLT polynomials indexed by \emph{circular} area sequences considered in \cite{AlexanderssonPanova2016}.

\subsection{Properties of LLT polynomials}

We use standard notation and let $\omega$ be the involution on symmetric functions 
that sends the complete homogeneous symmetric function $\completeH_\lambda$ to the elementary symmetric function
$\elementaryE_\lambda$, or equivalently, sends $\schurS_\lambda$ to $\schurS_{\lambda'}$. 

\begin{proposition}[{See \cite{AlexanderssonPanova2016}}]
 For any area sequence $\avec$ of length $n$,
 \begin{equation}\label{eq:omegaTransposed}
  \omega\LLT_{\avec}(\xvec;q) = q^{a_1+a_2+\dotsb+a_n}\LLT_{\avec^T}(\xvec;1/q)
 \end{equation}
where $\avec^T$ denotes the transpose of the Dyck diagram.
\end{proposition}

In \cite{AlexanderssonPanova2016}, we gave a proof that $\omega\LLT_{\avec,\svec}(\xvec;q+1)$
is positive in the power-sum basis. 
It also follows from a much more general result given in \cite{AlexanderssonSulzgruber2018}.
Note that if $f(\xvec)$ is $\elementaryE$-positive, then $\omega f(\xvec)$
is positive in the power-sum basis.
Later in \cref{prop:pexpansion}, the power-sum expansion of 
$\omega\LLT_{\avec,\svec}(\xvec;q+1)$ is stated explicitly.

\medskip 

The following lemma connects the LLT polynomials with the chromatic quasisymmetric
functions $\chrom_\avec(\xvec;q)$ introduced in \cite{ShareshianWachs2012}.
The function $\chrom_\avec(\xvec;q)$ is defined exactly as $\LLT_\avec(\xvec;q)$
but the sum in \cref{eq:LLTdef} is taken only over \emph{proper} colorings of $\Gamma_\avec$,
so that no monochromatic edges are allowed.

\begin{lemma}[{Adaptation of \cite[Prop. 3.5]{CarlssonMellit2017}. See also \cite[Sec. 5.1]{HaglundHaimanLoehr2005}}]\label{lem:plethRelation}
Let $\avec$ be a Dyck diagram. Then
\begin{align} \label{eq:plethRelation}
 (q-1)^{-n} \LLT_\avec[\xvec(q-1);q] =  \chrom_\avec(\xvec;q),
\end{align}
where the bracket denotes a substitution using plethysm.
\end{lemma}
From this formula, together with \cref{conj:main}, 
we have a novel conjectured formula for the chromatic quasisymmetric functions:
\begin{align}
\chrom_\avec(\xvec;q) = \sum_{\theta \in O(\avec)} (q-1)^{\asc(\theta)} 
\frac{\elementaryE_{\pi(\theta)}[\xvec(q-1)]}{(q-1)^n}.
\end{align}
Perhaps it is possible to do some sign-reversing involution together with plethysm manipulations
to obtain the $\elementaryE$-expansion of $\chrom_\avec(\xvec;q)$
and thus find a candidate formula for the Shareshian--Wachs conjecture.
\begin{conjecture}[{Shareshian--Wachs \cite{ShareshianWachs2012,ShareshianWachs2016}}]\label{conj:shareshianWachs}
There is some partition-valued statistic $\rho$ on \emph{acyclic orientations} of $\Gamma_\avec$, such that
\[
\chrom_\avec(\xvec;q) = \sum_{\theta \in AO(\avec)} q^{\asc(\theta)} \elementaryE_{\rho(\theta)}(\xvec).
\]
Here $AO(\avec)$ denotes the set of acyclic orientations of $\Gamma_\avec$.
\end{conjecture}
Note that the original Stanley--Stembridge conjecture is closely related to the $q=1$ case,
which was stated for the incomparability graphs of $3+1$-avoiding posets.

\begin{problem}
 Prove that the family $\LLTc_\avec(\xvec;q)$ defined in \eqref{eq:conjFormula} fulfills the 
 involution identity \eqref{eq:omegaTransposed}.
\end{problem}

\section{Recursive properties of LLT polynomials}\label{sec:recProp}

We shall now cover several recursive relations for the vertical-strip LLT polynomials.
Our proofs are bijective and directly use the combinatorial definition as a weighted sum over vertex colorings.
We illustrate these bijections with Dyck diagrams where 
only the relevant vertices and edges are shown.

The reader thus is encouraged to interpret a diagram as a weighted sum over colorings,
where decorations of the diagrams indicate restrictions of the colorings, or how the colorings contribute to $\asc(\cdot)$.
For example, given an edge $\epsilon$ of $\Gamma_{\avec,\svec}$, there are two possible cases.
Either $\epsilon$ contributes to the number of ascents, or it does not.
We can illustrate this simply as
\[
\begin{ytableau}
\;  &  *(yellow)\; \\
*(yellow)\;
\end{ytableau} 
=
 \begin{ytableau}
*(lightgray) \downarrow &  *(yellow)\; \\
*(yellow)\;
\end{ytableau}
+
q\;
\begin{ytableau}
*(lightgray) \rightarrow &  *(yellow)\; \\
*(yellow)\;
\end{ytableau}
\]
where the white box is the edge $\epsilon$ and $\downarrow$ indicates an edge that cannot be an ascent.
Note that the vertices shown do not need to have consecutive labels --- the intermediate vertices (and edges) are simply not shown.
Shaded boxes are not edges of $\Gamma_\avec$ and therefore does not contribute to ascents of the coloring.
To conclude, the class of diagrams considered here may be described as follows:
\begin{itemize}
 \item The white boxes are determined by some area sequence $\avec$, so that each white box is an edge in $\Gamma_\avec$.
 \item Every edge (box) is either white or shaded.
 \item Only white boxes contribute to the ascent statistic.
 \item A box (white or shaded) may contain an arrow, a $\rightarrow$ or $\downarrow$,
 imposing a strict or weak inequality requirement, respectively, on the colorings.
 In particular, a \emph{white} box containing a $\rightarrow$ is thus a sum over colorings where 
 this particular edge must be an ascent.
\end{itemize}
Note that this is a slightly broader class of diagrams than the class of vertical-strip diagrams, as the 
additional arrows impose more restrictions on the colorings.

\medskip

The following recursive relationship allows us to express vertical-strip LLT polynomials
as linear combinations of unicellular LLT polynomials.
Later in \cref{prop:verticalStripToUnicellularConj}, 
we prove that the polynomials in \cref{eq:conjFormula} satisfy the same recursion.
We use the notation $\avec\cup\{\epsilon\}$ to describe the area sequence of the unit interval graph
where the edge $\epsilon$ has been added to the edges of $\Gamma_\avec$.
The notation $\svec \cup \{\epsilon\}$ for strict edges is interpreted in a similar manner.

\begin{proposition}\label{prop:verticalStripToUnicellular}
If $\Gamma_{\avec,\svec}$ is a vertical strip diagram,
and $\epsilon$ is a non-strict outer corner of $\Gamma_{\avec,\svec}$, then
\begin{align}\label{eq:verticalStripToUnicellular}
 \LLT_{\avec\cup\{\epsilon\},\svec}(\xvec;q+1) = \LLT_{\avec,\svec}(\xvec;q+1) + q\LLT_{\avec, \svec \cup \{\epsilon\}}(\xvec;q+1).
\end{align}
\end{proposition}
\begin{proof}
By shifting the variable $q$, the identity can be restated as
\begin{align}\label{eq:verticalStripToUnicellular2}
 \LLT_{\avec\cup\{\epsilon\},\svec}(\xvec;q) + \LLT_{\avec, \svec \cup \{\epsilon\}}(\xvec;q)= q\LLT_{\avec, \svec \cup \{\epsilon\}}(\xvec;q) + \LLT_{\avec,\svec}(\xvec;q),
\end{align}
which in (as sum over colorings) diagram form can be expressed as follows.
The two vertices shown are the vertices of $\epsilon$.
\[
\ytableausetup{boxsize=1.0em}
\begin{ytableau}
\; &  *(yellow)\;\\
*(yellow) \;
\end{ytableau}
+
\begin{ytableau}
*(lightgray) \rightarrow & *(yellow) \; \\
*(yellow) \;
\end{ytableau}
=
q\;
\begin{ytableau}
*(lightgray) \rightarrow & *(yellow)\;  \\
*(yellow)\;
\end{ytableau}
+
\begin{ytableau}
*(lightgray) &  *(yellow)\; \\
*(yellow)\;
\end{ytableau}
\]
The first and last diagram can be expanded into subcases,
\[
\ytableausetup{boxsize=0.8em}
\left(
\begin{ytableau}
*(lightgray) \downarrow &  *(yellow)\; \\
*(yellow)\;
\end{ytableau}
+
q\;
\begin{ytableau}
*(lightgray) \rightarrow &  *(yellow)\; \\
*(yellow)\;
\end{ytableau}
\right)
+
\begin{ytableau}
*(lightgray) \rightarrow & *(yellow)\; \\
*(yellow)\;
\end{ytableau}
=
q\;
\begin{ytableau}
*(lightgray) \rightarrow & *(yellow)\; \\
*(yellow)\;
\end{ytableau}
+
\left(
\begin{ytableau}
*(lightgray) \downarrow &  *(yellow)\;\\
*(yellow)\;
\end{ytableau}
+
\begin{ytableau}
*(lightgray) \rightarrow &  *(yellow)\;\\
*(yellow)\;
\end{ytableau}
\right)
\]
and here it is evident that both sides agree.
\end{proof}
The above recursion seem to relate to certain recursions on Catalan symmetric functions, 
see \cite[Prop. 5.6]{BlasiakMorsePunSummers2018}.
Catalan symmetric functions are very similar in nature to LLT polynomials.

\subsection{Lee's recursion}

In \cref{prop:verticalStripRec} below, we prove a recursion on certain LLT polynomials. 
We then show that this relation is equivalent to Lee's recursion, given in \cite[Thm 3.4]{Lee2018}.

\begin{definition}
Let $\avec$ be an area sequence of length $n\geq 3$. 
An edge $(i,j) \in E(\Gamma_\avec)$, $3\leq j \leq n$, is said to be \defin{admissible} if 
the following four conditions hold:
\begin{itemize}
 \item $i = j-a_j$
 \item $j=n$ or $a_j \geq a_{j+1} +1$
 \item $a_j \geq 2$,
 \item $a_i+1 = a_{i+1}$.
\end{itemize}
The last condition is automatically satisfied if the first three are true
and $\avec$ is abelian.
Note that if $(i,j)$ is admissible,
then for all $k<i$ or $k>i+1$ we have 
\begin{align}\label{eq:sameNumberOfEdges}
(k,i) \in E(\Gamma_\avec) \Leftrightarrow (k,i+1) \in E(\Gamma_\avec) 
 \text{ and }
 (i,k) \in E(\Gamma_\avec) \Leftrightarrow (i+1,k) \in E(\Gamma_\avec).
\end{align}
These properties are crucial in later proofs.
\end{definition}

\begin{example}
For the following diagram $\avec$, the edge $(2,5)$ is admissible. 
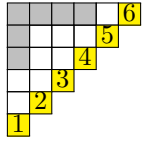
\[
 \begin{ytableau}
*(lightgray)  & *(lightgray) & *(lightgray) & *(lightgray) & \; & *(yellow) 6 \\
*(lightgray) & \; & \; & \; & *(yellow) 5 \\
*(lightgray) & \; & \; & *(yellow) 4 \\
 \; & \; & *(yellow) 3 \\
 \; & *(yellow) 2 \\
 *(yellow) 1
\end{ytableau}
\]
\end{example}
\medskip 

Let $\evec_j$ denote the $j$th unit vector.
\begin{proposition}\label{prop:verticalStripRec}
Suppose $(i,j)$ is an admissible edge of the area sequence $\avec$,
set $\avec^1 \coloneqq \avec-\evec_j$ and $\avec^2 \coloneqq \avec-2\evec_j$ 
and $\svec^1 \coloneqq \{(i,j)\}$, $\svec^2 \coloneqq \{(i+1,j)\}$.
Then 
\begin{align}\label{eq:verticalStripRec}
 \LLT_{\avec^1,\svec^1}(\xvec;q) = q\LLT_{\avec^2,\svec^2}(\xvec;q).
\end{align}
\end{proposition}
\begin{proof}
We use the diagram-type proof as before, now only showing the vertices $i$, $i+1$ and $j$.
The identity we wish to show is then presented as
\begin{align*}
\begin{ytableau}
*(lightgray) \rightarrow &  \; &  *(yellow)\;\\
\; &  *(yellow)\; \\
*(yellow)\;
\end{ytableau} &=
q \; \begin{ytableau}
*(lightgray) & *(lightgray) \rightarrow &  *(yellow)\;\\
\; &  *(yellow)\; \\
*(yellow)\;
\end{ytableau}.
\end{align*}
Both sides are considered as a weighted sum over colorings with restrictions indicated by $\to$. 
Subdividing these sums into subcases by forcing additional inequalities gives
\begin{align*}
q\;
\begin{ytableau}
*(lightgray) \rightarrow &  *(lightgray) \rightarrow &  *(yellow)\;\\
 \; &  *(yellow)\; \\
 *(yellow)\;
\end{ytableau}
+
\begin{ytableau}
*(lightgray) \rightarrow & *(lightgray) \downarrow &  *(yellow)\;\\
\; &  *(yellow)\; \\
*(yellow)\;
\end{ytableau} 
&=
q \; \left( 
\begin{ytableau}
*(lightgray) \rightarrow & *(lightgray) \rightarrow &  *(yellow)\;\\
\; &  *(yellow)\; \\
*(yellow)\;
\end{ytableau}
+
\begin{ytableau}
*(lightgray) \downarrow & *(lightgray) \rightarrow &  *(yellow)\;\\
\; &  *(yellow)\; \\
*(yellow)\;
\end{ytableau}
\right).
\end{align*}
Two terms cancel and additional inequalities follows by transitivity.
It therefore suffices to prove the following. 
\[
q\;
\begin{ytableau}
*(lightgray) \rightarrow &  *(lightgray) \downarrow &  *(yellow)\;\\
*(lightgray) \rightarrow &  *(yellow)\; \\
*(yellow)\;
\end{ytableau}
=
q\;
\begin{ytableau}
*(lightgray) \downarrow & *(lightgray) \rightarrow &  *(yellow)\;\\
*(lightgray) \downarrow &  *(yellow)\; \\
*(yellow)\;
\end{ytableau}
\]
Note that the additional $q$ in the left hand side appears due to the ascent $(i,i+1)$.

There is now a simple $q$-weight-preserving bijection between colorings on the diagram on the left hand side,
and colorings of the diagram on the right hand side.
For a coloring $\kappa$ in the left hand side, we have $\kappa(i) < \kappa(j)\leq \kappa(i+1)$,
while on the right hand side, we have $\kappa(i+1) < \kappa(j) \leq \kappa(i)$.
Hence, vertex $i$ and vertex $i+1$ are never assigned the same color.

The bijection is to simply swap the colors of the adjacent vertices $i$ and $i+1$.
The property in \cref{eq:sameNumberOfEdges} ensures that the number 
of ascending edges are preserved under this swap.
\end{proof}

\medskip 

\begin{corollary}[Local linear relation {\cite[Thm 3.4]{Lee2018}}]\label{prop:lee}
Let $\avec$ be an area sequence for which $(i,j)$ is admissible,
and let $\avec^0 \coloneqq \avec$, $\avec^1 \coloneqq \avec-\evec_j$ and $\avec^2 \coloneqq \avec-2\evec_j$.
Then 
\begin{align}\label{eq:leeFormula}
 \LLT_{\avec^0}(\xvec;q) - \LLT_{\avec^1}(\xvec;q) = 
 q(\LLT_{\avec^1}(\xvec;q) - \LLT_{\avec^2}(\xvec;q)).
\end{align}
\end{corollary}
\begin{proof}
We see that the left hand side of \eqref{eq:leeFormula} can be rewritten in diagram form using 
\cref{eq:verticalStripToUnicellular}:
\[
\textrm{LHS}=
\begin{ytableau}
\; &  \; &  *(yellow)\;\\
\; &  *(yellow)\; \\
*(yellow)\;
\end{ytableau}
-
\begin{ytableau}
*(lightgray) &  \; &  *(yellow)\;\\
\; &  *(yellow)\; \\
*(yellow)\;
\end{ytableau}
=
(q-1)
\begin{ytableau}
*(lightgray) \to &  \; &  *(yellow)\;\\
\; &  *(yellow)\; \\
*(yellow)\;
\end{ytableau}
\]
The right hand side is treated in a similar manner:
\[
\textrm{RHS}=
q\left(
\begin{ytableau}
*(lightgray) &  \; &  *(yellow)\;\\
\; &  *(yellow)\; \\
*(yellow)\;
\end{ytableau}
-
\begin{ytableau}
*(lightgray) &  *(lightgray) &  *(yellow)\;\\
\; &  *(yellow)\; \\
*(yellow)\;
\end{ytableau}
\right)
=
q(q-1)
\begin{ytableau}
*(lightgray)  &  *(lightgray) \to &  *(yellow)\;\\
\; &  *(yellow)\; \\
*(yellow)\;
\end{ytableau}
\]
The identity in \eqref{eq:verticalStripRec} now implies that $\textrm{LHS}=\textrm{RHS}$.
\end{proof}

\begin{example}
As an illustration of \cref{prop:lee}, we have $(i,j)=(2,5)$ and the following 
three Dyck diagrams.
\[
\avec^0=
\begin{ytableau}
*(lightgray)  & *(lightgray) & *(lightgray) & *(lightgray) & \; & *(yellow) 6 \\
*(lightgray) & \; & \; & \; & *(yellow) 5\\
*(lightgray) & \; & \; & *(yellow) 4\\
 \; & \; & *(yellow) 3\\
 \; & *(yellow) 2\\
 *(yellow) 1
\end{ytableau}
\quad 
\avec^1=
\begin{ytableau}
*(lightgray)  & *(lightgray) & *(lightgray) & *(lightgray) & \; & *(yellow) 6\\
*(lightgray) & *(lightgray) & \; & \; & *(yellow) 5\\
*(lightgray) & \; & \; & *(yellow)4 \\
 \; & \; & *(yellow) 3\\
 \; & *(yellow) 2\\
 *(yellow) 1
\end{ytableau}
\quad 
\avec^2=
\begin{ytableau}
*(lightgray)  & *(lightgray) & *(lightgray) & *(lightgray) & \; & *(yellow)6 \\
*(lightgray) & *(lightgray) & *(lightgray) & \; & *(yellow) 5\\
*(lightgray) & \; & \; & *(yellow) 4\\
 \; & \; & *(yellow)3 \\
 \; & *(yellow) 2\\
 *(yellow) 1
\end{ytableau}
\]
\end{example}

\subsection{The dual Lee recursion}

There is a ``dual'' version of \cref{prop:lee},
obtained by applying $\omega$ to both sides of \eqref{eq:leeFormula},
and then apply the relation in \eqref{eq:omegaTransposed}.
We shall now state this in more detail.

\begin{definition}
Let $\avec$ be an area sequence of length $n\geq 3$.
An edge $(i,j)$ is said to be \emph{dual-admissible} if the edge $(n+1-j,n+1-i)$
is admissible for $\avec^T$.
\end{definition}

We can then formulate the dual versions of \cref{prop:verticalStripRec}
and \cref{prop:lee}.
\begin{proposition}[The dual Lee recursion]\label{prop:dualLee}
Let $\avec$ be an area sequence for which $(i,j)$ is dual-admissible
and let $\avec^0 \coloneqq \avec$, $\avec^1 \coloneqq \avec-\evec_j$ and $\avec^2 \coloneqq \avec-\evec_j-\evec_{j-1}$.
Then 
\begin{align}\label{eq:dualVerticalStripRec}
 \LLT_{\avec^1,\svec^1}(\xvec;q) = q\LLT_{\avec^2,\svec^2}(\xvec;q)
\end{align}
and
\begin{align}\label{eq:dualLeeFormula}
 \LLT_{\avec^0}(\xvec;q) - \LLT_{\avec^1}(\xvec;q) = 
 q(\LLT_{\avec^1}(\xvec;q) - \LLT_{\avec^2}(\xvec;q))
\end{align}
where $\svec^1 \coloneqq \{(i,j)\}$ and $\svec^2 \coloneqq \{(i,j-1)\}$.
\end{proposition}
\begin{proof}[Proof sketch]
We can either prove these identities by applying $\omega$ as outlined above, 
or bijectively using diagrams. We leave out the details.
\end{proof}

\begin{example}
\cref{prop:verticalStripRec} applies in the following generic situation.
Here, the edge $(x,z)$ is an admissible edge.
The crucial condition in \eqref{eq:sameNumberOfEdges} states that
the area of the rows with vertices $x$ and $y$ in the diagram differ by exactly one.
\begin{align}\label{eq:leeDiagrams}
(\avec^1,\svec^1)=
\begin{tikzpicture}[baseline=(current bounding box.center)]
\draw[step=1em, gray, very thin] (-0.001,0) grid (9em, 9em);
\fill[gray,x=1em,y=1em] (5,8)--(5,9)--(0,9)--(0,0)--(1,0)--(1,7)--(3,7)--(3,8)--(5,8);
\fill[gray,x=1em,y=1em] (3.01,7.99)--(3.99,7.99)--(3.99,7.01)--(3.01,7.01);
%
\draw[black,thick,x=1em,y=1em] (3,2)--(4,2)--(4,3)--(5,3)--(5,4)--(6,4)--(6,5)--(7,5)--(7,6)--(8,6)--(8,7)--(9,7)--(9,8);
\draw[black,thick,x=1em,y=1em] (3,2)--(3,1)--(2,1)--(2,0);
\draw[black, thick,x=1em,y=1em] (3,7)--(3,8)--(5,8);
\draw[black, thick,x=1em,y=1em] (1,2)--(1,4);
\node[x=1em,y=1em] (1) at (3.5, 7.5) {$\rightarrow$};
\node[x=1em,y=1em] (3) at (3.5, 2.5) {x};
\node[x=1em,y=1em] (4) at (4.5, 3.5) {y};
\node[x=1em,y=1em] (5) at (8.5, 7.5) {z};
\end{tikzpicture}
\qquad
(\avec^2,\svec^2)=
\begin{tikzpicture}[baseline=(current bounding box.center)]
\draw[step=1em, gray, very thin] (-0.001,0) grid (9em, 9em);
\fill[gray,x=1em,y=1em] (5,8)--(5,9)--(0,9)--(0,0)--(1,0)--(1,7)--(3,7)--(3,8)--(5,8);
\fill[gray,x=1em,y=1em] (3.01,7.99)--(3.99,7.99)--(3.99,7.01)--(3.01,7.01);
\fill[gray,x=1em,y=1em] (4.01,7.99)--(4.99,7.99)--(4.99,7.01)--(4.01,7.01);
%
%
\draw[black, thick,x=1em,y=1em] (3,2)--(4,2)--(4,3)--(5,3)--(5,4)--(6,4)--(6,5)--(7,5)--(7,6)--(8,6)--(8,7)--(9,7)--(9,8);
\draw[black,thick,x=1em,y=1em] (3,2)--(3,1)--(2,1)--(2,0);
\draw[black, thick,x=1em,y=1em] (3,7)--(3,8)--(5,8);
\draw[black, thick,x=1em,y=1em] (1,2)--(1,4);
\node[x=1em,y=1em] (2) at (4.5, 7.5) {$\rightarrow$};
\node[x=1em,y=1em] (3) at (3.5, 2.5) {x};
\node[x=1em,y=1em] (4) at (4.5, 3.5) {y};
\node[x=1em,y=1em] (5) at (8.5, 7.5) {z};
\end{tikzpicture}
\end{align}
Similarly, the dual recursion in \cref{eq:dualVerticalStripRec} applies 
in the following situation, where $(x,z)$ is a dual-admissible edge:
\begin{align}\label{eq:dualLeeDiagrams}
(\avec^1,\svec^1)=
\begin{tikzpicture}[baseline=(current bounding box.center)]
\begin{scope}[yscale=-1,rotate=90]
\draw[step=1em, gray, very thin] (-0.001,0) grid (9em, 9em);
\fill[gray,x=1em,y=1em] (5,8)--(5,9)--(0,9)--(0,0)--(1,0)--(1,7)--(3,7)--(3,8)--(5,8);
\fill[gray,x=1em,y=1em] (3.01,7.99)--(3.99,7.99)--(3.99,7.01)--(3.01,7.01);
%
\draw[black,thick,x=1em,y=1em] (3,2)--(4,2)--(4,3)--(5,3)--(5,4)--(6,4)--(6,5)--(7,5)--(7,6)--(8,6)--(8,7)--(9,7)--(9,8);
\draw[black,thick,x=1em,y=1em] (3,2)--(3,1)--(2,1)--(2,0);
\draw[black, thick,x=1em,y=1em] (3,7)--(3,8)--(5,8);
\draw[black, thick,x=1em,y=1em] (1,2)--(1,4);
\node[x=1em,y=1em] (1) at (3.5, 7.5) {$\rightarrow$};
\node[x=1em,y=1em] (3) at (3.5, 2.5) {z};
\node[x=1em,y=1em] (4) at (4.5, 3.5) {y};
\node[x=1em,y=1em] (5) at (8.5, 7.5) {x};
\end{scope}
\end{tikzpicture}
\qquad
(\avec^2,\svec^2)=
\begin{tikzpicture}[baseline=(current bounding box.center)]
\begin{scope}[yscale=-1,rotate=90]
\draw[step=1em, gray, very thin] (-0.001,0) grid (9em, 9em);
\fill[gray,x=1em,y=1em] (5,8)--(5,9)--(0,9)--(0,0)--(1,0)--(1,7)--(3,7)--(3,8)--(5,8);
\fill[gray,x=1em,y=1em] (3.01,7.99)--(3.99,7.99)--(3.99,7.01)--(3.01,7.01);
\fill[gray,x=1em,y=1em] (4.01,7.99)--(4.99,7.99)--(4.99,7.01)--(4.01,7.01);
%
%
\draw[black, thick,x=1em,y=1em] (3,2)--(4,2)--(4,3)--(5,3)--(5,4)--(6,4)--(6,5)--(7,5)--(7,6)--(8,6)--(8,7)--(9,7)--(9,8);
\draw[black,thick,x=1em,y=1em] (3,2)--(3,1)--(2,1)--(2,0);
\draw[black, thick,x=1em,y=1em] (3,7)--(3,8)--(5,8);
\draw[black, thick,x=1em,y=1em] (1,2)--(1,4);
\node[x=1em,y=1em] (2) at (4.5, 7.5) {$\rightarrow$};
\node[x=1em,y=1em] (3) at (3.5, 2.5) {z};
\node[x=1em,y=1em] (4) at (4.5, 3.5) {y};
\node[x=1em,y=1em] (5) at (8.5, 7.5) {x};
\end{scope}
\end{tikzpicture}
\end{align}
\end{example}

\subsection{Recursion in the complete graph case}

We end this section by recalling a recursion for 
LLT polynomials in the complete graph case.

\begin{proposition}[{\cite[Prop.5.8]{AlexanderssonPanova2016}}]\label{prop:completeGraphLLTRecursion}
Let $\LLT_{K_n}(\xvec;q)$ denote the LLT polynomial for the complete graph,
where the area sequence is $(0,1,2,\dotsc,n-1)$.
Then
\begin{align}\label{eq:lltCompleteGraphRecurrence}
 \LLT_{K_n}(\xvec;q) = \sum_{i=0}^{n-1} \LLT_{K_i}(\xvec;q) e_{n-i}(\xvec) 
 \prod_{k=i+1}^{n-1} \left[ q^{k} -1 \right]
 ,\; \LLT_{K_0}(\xvec;q) = 1.
\end{align}
\end{proposition}

\begin{lemma}\label{lem:rectangularLLT}
If $\avec$ is rectangular and the non-edges form a $k\times (n-k)$-rectangle in the Dyck diagram,
then $\LLT_{\avec}(\xvec;q) = \LLT_{K_k}(\xvec;q)\LLT_{K_{n-k}}(\xvec;q)$.
\end{lemma}
\begin{proof}
 The unit-interval graph $\Gamma_\avec$ is a disjoint union of two smaller complete graphs,
 so this now follows immediately from the definition of unicellular LLT polynomials.
\end{proof}

For the remaining of this section, it will be more convenient to use the notation in
\cite{Lee2018}, and index unicellular LLT polynomials of degree $n$ with partitions $\lambda$ that fit inside the 
staircase $(n-1,n-2,\dotsc,2,1,0)$. We fix $n$ and let the 
area sequence $\avec$ correspond to the partition $\lambda$ where $\lambda_i = n-i - a_{n+1-i}$.
Hence, $\lambda$ is exactly the shape of the (shaded) non-edges in the Dyck diagram.
By definition, $\lambda$ is abelian if it fits inside some $k \times (n-k)$-rectangle.

\begin{lemma}[{Follows from \cite[Thm. 3.4]{HuhNamYoo2018}}]\label{lem:rowNumberReduction}
Let $\lambda$ be abelian with $\ell\geq 2$ parts such that $\lambda_\ell<\lambda_{\ell-1}$.
Let
\begin{align*}
 \mu = (\lambda_1,\lambda_2,\dotsc,\lambda_{\ell-1})\quad \text{ and } \quad
 \nu = (\lambda_1,\lambda_2,\dotsc,\lambda_{\ell-1},\lambda_{\ell}+1).
\end{align*}
Then there are rational functions $c(q)$ and $d(q)$ such that $\LLT_\lambda(\xvec;q) = c(q)\LLT_\mu(\xvec;q) + d(q)\LLT_\nu(\xvec;q)$.
\end{lemma}
\begin{proof}
Use \cref{prop:lee} repeatedly on row $\ell$ of $\mu$.
\end{proof}

\begin{example}
 To illustrate \cref{lem:rowNumberReduction}, we have the following three partitions:
 \[
  \lambda = \ydiagram{6,5,3}\quad \mu = \ydiagram{6,5,0}, \quad \nu = \ydiagram{6,5,4}
 \]
\end{example}

\begin{proposition}\label{prop:recToAbelian}
Every $\LLT_{\lambda}(\xvec;q)$ where $\lambda$ is abelian,
can be expressed recursively via Lee's recurrences,
as a linear combination of some $\LLT_{\mu^j}(\xvec;q)$ where the $\mu^j$ are rectangular.
\end{proposition}
\begin{proof}
Let $\lambda$ be abelian partition with exactly $\ell$ parts, so that it fits in a $\ell \times (n-\ell)$-rectangle.
We shall do a proof by induction over $\lambda$, and in particular its number of parts.

\begin{enumerate}
 \item 
\textbf{Case $\lambda=\emptyset$.}
This is rectangular by definition.

\item 
\textbf{Case $\lambda=(n-1)$.} This is rectangular.

\item 
\textbf{Case $\ell=1$.} Use \cref{lem:rowNumberReduction}
to reduce to Case (1) and Case (2).

\item 
\textbf{Case $\ell>1$ and $\lambda_i \leq \ell-i$ for some $i \in [\ell]$.}
The conditions imply that it is possible to 
remove a $2\times 1$ or a $1 \times 2$-domino from $\lambda$ and obtain a new partition.
Hence we can use Lee's recursion to express $\LLT_{\lambda}(\xvec;q)$ using polynomials indexed by two smaller partitions.
For example, this case applies in the following situation:
\begin{equation}
 \lambda = \ydiagram{4,2,2,1,1}\quad \longrightarrow \quad  \ydiagram{4,2,1,1,1} 
 \quad \text{ and } \quad  \ydiagram{4,1,1,1,1}
\end{equation}

\item 
\textbf{Case $\ell>1$ and $\lambda_i > \ell-i$ for all $i \in [\ell]$.}
Three things can happen here, and it is easy to see that this list is exhaustive.
\begin{itemize}
 \item $\lambda$ is rectangular and we are done.
 \item We can \emph{add} a $2\times 1$ or $1 \times 2$-domino to $\lambda$ without increasing $\ell$ and still obtain a partition.
 Similar to Case (4), we can therefore reduce to cases where $|\lambda|$ has \emph{increased} by $1$ and $2$.
 \item \cref{lem:rowNumberReduction} can be applied, thus reducing $\lambda$ to a case where $\ell$ has strictly been decreased,
 and a case where $\lambda$ has increased by one box.
\end{itemize}
\end{enumerate}
Notice that Case (4) reduces \emph{only} back to Case (4), or a case where $\ell$ is decreased,
and the same goes for Case (5). There are therefore no circular dependencies 
amongst these cases and the induction is valid.
\end{proof}

\section{Recursions for the conjectured formula}\label{sec:recPropConj}

In this section, we prove that $\LLTc(\xvec;q)$ also fulfills the recursion in
\cref{prop:verticalStripToUnicellular}.
We use similar bijective technique as in \cref{sec:recProp},
but \emph{diagrams now represent weighted sums over orientations} as in 
\cref{eq:conjFormula}. Note that we now also consider the shifted polynomial $\LLTc_{\avec,\svec}(\xvec;q+1)$.

Each diagram now represents a weighted sum over orientations, 
where the weight of a single orientation $\theta$ 
is $q^{\asc(\theta)}\elementaryE_{\hrvp(\theta)}$.
The class of diagrams we now consider is as follows.
\begin{itemize}
 \item The white boxes are determined by some area sequence.
 \item Every edge (box) is either white or shaded.
 \item Only white boxes contribute to the ascent statistic.
 \item A box (white or shaded) may contain an arrow, a $\rightarrow$ or $\downarrow$,
 imposing a restriction on the orientations being summed over.
 In particular, a \emph{white} box containing a $\rightarrow$ is thus a sum over orientations where 
 this particular edge must be an ascent.
\end{itemize}
Hence, each diagram represents a sum over $2^W$ orientations, where $W$ is the number of white boxes not containing any arrow.

\begin{example}
Suppose the following diagram illustrates the entire graph.
The diagram represents the weighted sum over all orientations of the non-specified 
edges $(x,y)$ and $(y,z)$. 
The edge $(x,z)$ is strict, and $(z,w)$ is forced to be ascending. 
Remember that each ascending edge contributes with a $q$-factor.
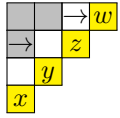
\[
 \ytableausetup{boxsize=1.0em}
\begin{ytableau}
*(lightgray)  & *(lightgray) & \to & *(yellow) w\\
*(lightgray) \to &  & *(yellow) z \\
\; &  *(yellow) y \\
*(yellow) x
\end{ytableau}
\]
There are four orientations in total,
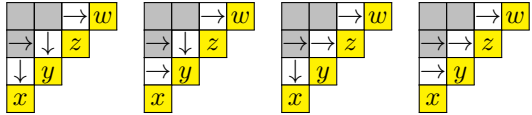
\[
\begin{ytableau}
*(lightgray)  & *(lightgray) & \to & *(yellow) w\\
*(lightgray) \to & \downarrow & *(yellow) z \\
\downarrow &  *(yellow) y \\
*(yellow) x
\end{ytableau}
\quad 
\begin{ytableau}
*(lightgray)  & *(lightgray) & \to & *(yellow) w\\
*(lightgray) \to & \downarrow & *(yellow) z \\
\to &  *(yellow) y \\
*(yellow) x
\end{ytableau}
\quad 
\begin{ytableau}
*(lightgray)  & *(lightgray) & \to & *(yellow) w\\
*(lightgray) \to & \to & *(yellow) z \\
\downarrow &  *(yellow) y \\
*(yellow) x
\end{ytableau}
\quad 
\begin{ytableau}
*(lightgray)  & *(lightgray) & \to & *(yellow) w\\
*(lightgray) \to & \to & *(yellow) z \\
\to &  *(yellow) y \\
*(yellow) x
\end{ytableau}
\]
which according to \eqref{eq:conjFormula}
give the sum $q\elementaryE_{31} + q^2\elementaryE_{31} + q^2\elementaryE_{4} + q^3\elementaryE_{4}$.
\end{example}
In the diagrams below, only relevant vertices of the graphs are included.

\begin{proposition}\label{prop:verticalStripToUnicellularConj}
If $\Gamma_{\avec,\svec}$ is a vertical-strip graph, with $\epsilon$ being a non-strict outer corner,
then 
\begin{align}\label{eq:verticalStripToUnicellularConj}
 \LLTc_{\avec\cup\{\epsilon\},\svec}(\xvec;q+1)=  q\LLTc_{\avec, \svec \cup \{\epsilon\}}(\xvec;q+1)+ \LLTc_{\avec,\svec}(\xvec;q+1).
\end{align}
\end{proposition}
\begin{proof}
In diagram form, this amounts to showing that orientations of the diagram in the left hand side can be put in $q$-weight-preserving 
bijection with the disjoint sets of orientations in the right hand side, 
while also preserving the $\hrvp(\cdot)$-statistic.
Thus we wish to prove that
\[
\ytableausetup{boxsize=1.0em}
\begin{ytableau}
\; &  *(yellow) y\\
*(yellow) x
\end{ytableau}
=
q\;
\begin{ytableau}
*(lightgray) \rightarrow & *(yellow) y  \\
*(yellow) x
\end{ytableau}
+
\begin{ytableau}
*(lightgray) &  *(yellow) y \\
*(yellow) x
\end{ytableau}.
\]
Consider an orientation in the left hand side. There are two cases to consider:
\begin{itemize}
 \item The edge $(x,y)$ is ascending. We map the orientation to an orientation of the first diagram in the right hand side,
 by preserving the orientation of all other edges.
 \item The edge $(x,y)$ is non-ascending. We map this to the second diagram, 
 by preserving the orientation of all other edges.
\end{itemize}
In both cases above, note that both the $q$-weight 
and $\hrvp(\cdot)$ is preserved under this map.
\end{proof}

\begin{corollary}\label{cor:unicellularImpliesVS}
If \cref{conj:main} holds in the unit-interval case, it holds in the vertical-strip case.
\end{corollary}
\begin{proof}
 We can use \cref{prop:verticalStripToUnicellularConj} and \cref{prop:verticalStripToUnicellular}
 to recursively remove all strict edges. Since both families satisfy the same recursion,
 we have that the unicellular case of \cref{conj:main} implies the vertical-strip case.
\end{proof}

\subsection{The complete graph recursion and line graphs}

Analogous to \cref{prop:completeGraphLLTRecursion}, we have a recursion for the
corresponding $\LLTc_{K_n}(\xvec;q)$, where we again consider the complete graph case.
Here, $\binom{[n]}{k}$ denotes the set of $k$-subsets of $\{1,\dotsc,n\}$.
\begin{lemma}\label{lem:completeGraphLLTcRecursion}
The polynomial $\LLTc_{K_n}(\xvec;q)$ satisfy $\LLTc_{K_0}(\xvec;q) \coloneqq 1$
and $\LLTc_{K_n}(\xvec;q+1)$ is equal to
\begin{align}\label{eq:lltCompleteGraphRecurrence2}
 \sum_{i=0}^{n-1} \LLTc_{K_i}(\xvec;q+1) \elementaryE_{n-i}(\xvec) 
 \left( \sum_{S\in \binom{[n-1]}{n-1-i}} \prod_{j=1}^{n-1-i} (q+1)^{s_j - j}[ (q+1)^j - 1 ] \right).
\end{align}
\end{lemma}
\begin{proof}
We first give an argument for the recursion in \eqref{eq:lltCompleteGraphRecurrence2}.
Given an orientation $\theta$ of $K_i$, 
we can construct a new orientation $\theta'$ of $K_n$ by inserting a new part of size $n-i$
in the vertex partition where vertex $n$ is a member.
Choose an $i$-subset of $[n-1]$ and let $\theta$ define the orientation of the edges in $\theta'$
on these vertices. The remaining $n-i-1$ vertices will be in the new part ---
let us call this set of vertices $S = \{s_1,\dotsc,s_{n-i-1}\}$
where $n>s_1>s_2 > \dotsb > s_{n-i-1} \geq 1$.
Each element $s_j$ must have at least one ascending edge to either vertex $n$,
or to another member in $S$ larger than $s_j$, 
but all other choices of ascending edges are allowed.
It then follows that  that for such a subset $S$, there are
\[
 \prod_{j=1}^{n-1-i} (q+1)^{n-s_j-j}[ (q+1)^j - 1 ]
\]
$\asc(\cdot)$-weighted ways of choosing subsets of ascending edges in $\theta'$ so that all vertices in $S$
has $n$ as highest reachable vertex.
Hence,
\[
 \sum_{S\in \binom{[n-1]}{n-1-i}} \prod_{j=1}^{n-1-i} (q+1)^{n-s_j - j}[ (q+1)^j - 1 ]
\]
is the $\asc(\cdot)$-weighted count of the number of orientations of $K_n$,
where the part of the vertex-partition containing $n$ has exactly $n-i$ members.
Finally, by sending each $s_i$ to $n-s_i$, which is an involution on $\binom{[n-1]}{n-1-i}$,
we get the desired formula.
\end{proof}

We shall now prove that $\LLTc_{K_n}(\xvec;q) = \LLT_{K_n}(\xvec;q)$.
By using \cref{lem:completeGraphLLTcRecursion} and \cref{prop:completeGraphLLTRecursion},
this follows from the following lemma.
\begin{lemma}
For all $n$ and $1\leq i \leq n-1$, we have that
\[
 \prod_{k=i+1}^{n-1} \left[ q^{k} -1 \right] = 
 \sum_{S\in \binom{[n]}{n-i-1}} \prod_{j=1}^{n-1-i} q^{s_j - j}[ q^j - 1 ].
\]
\end{lemma}
\begin{proof}
A small rewrite in each of the product indices gives
\[
 \prod_{k=1}^{(n-i-1)} \left[ q^{k+1} -1 \right] = 
 \sum_{S\in \binom{[n]}{(n-i-1)}} \prod_{j=1}^{(n-i-1)} q^{s_j}[1- q^{-j}].
\]
We may now substitute $i \coloneqq (n-i-1)$ and it suffices to prove that
\[
 \prod_{k=1}^{i} \left[ q^{n-k+1}-1 \right] = \sum_{S\in \binom{[n]}{i}} \prod_{j=1}^{i} q^{s_j}[ 1 - q^{-j}].
\]
This can be restated as
\[
 \prod_{k=1}^{i} \frac{ q^{n+1} - q^k }{ q^k - 1 } = \sum_{S\in \binom{[n]}{i}} q^{|S|}
\]
where $|S|$ denotes the sum of the entries in $S$. 
We can subdivide the right hand sum depending on if $n \in S$ or not,
\[
 \sum_{S\in \binom{[n]}{i}} q^{|S|} = \sum_{S\in \binom{[n-1]}{i}} q^{|S|} \quad  + \quad  
 q^n \sum_{S\in \binom{[n-1]}{i-1}} q^{|S|}.
\]
By induction over $n$ and $i$ it suffices to show that 
\[
 \prod_{k=1}^{i} \frac{ q^{n+1} - q^k }{ q^k - 1 } = 
 \prod_{k=1}^{i} \frac{ q^{n} - q^k }{ q^k - 1 } + 
 q^n \prod_{k=1}^{i-1} \frac{ q^{n} - q^k }{ q^k - 1 }
\]
and this is easy to verify.
\end{proof}

The case of line graphs follows immediately from \cite[Prop. 5.18]{AlexanderssonPanova2016}.
\begin{proposition}\label{prop:linegraph}
Let $\avec = (0,1,1,\dotsc,1)$ be a line graph. 
Then $\LLTc_{\avec}(\xvec;q) = \LLT_{\avec}(\xvec;q)$.
\end{proposition}

\subsection{On Lee's recursion for orientations}

We would also like to prove that the $\LLTc(\xvec;q)$ fulfill Lee's recursions.
However, this is a surprisingly challenging and we are unable to show this at the present time.
A proof that Lee's recursions hold for $\LLTc(\xvec;q)$ would 
imply that $\LLT_\avec(\xvec;q) = \LLTc_\avec(\xvec;q)$ 
at least for all abelian area sequences $\avec$.
Computer experiment with $n\leq 7$ confirms that the polynomials $\LLTc_\avec(\xvec;q)$
indeed do satisfy these recurrences.
\medskip 

The class of melting lollipop graphs
can be constructed recursively from the complete graphs and the line graphs 
by just using the recursion in \cref{prop:lee}.
This is in fact done in \cite{HuhNamYoo2018}, so we simply sketch a proof of this statement.
Recall that a melting lollipop graph $\avec$ is given by
 \[
  a_i = \begin{cases}
	i-1 \text{ for } i=1,\dotsc,m-1 \\
	m-1-k \text{ for } i=m \\
	1\text{ for } i=m+1,\dotsc,m+n
\end{cases}
 \]
for some $m,n\geq 1$ and $0\leq k \leq m-1$.
 Melting lollipop graphs for various parameters are shown below.
 \[
\ytableausetup{boxsize=0.4em}
A=
\substack
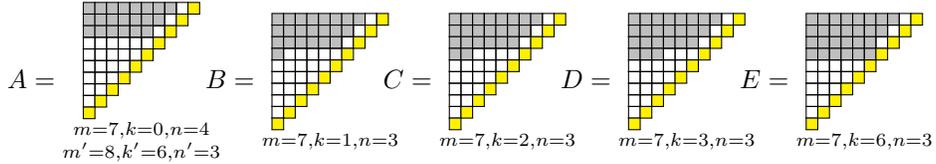
{
\begin{ytableau}
 *(lightgray)  &*(lightgray)  &*(lightgray)  &*(lightgray)  &*(lightgray)  & *(lightgray)  & *(lightgray)  & *(lightgray)  & \;  & *(yellow)  \\
 *(lightgray)  &*(lightgray)  &*(lightgray)  &*(lightgray)  & *(lightgray)  & *(lightgray)  & *(lightgray)  & \;  & *(yellow)  \\
 *(lightgray)  &*(lightgray)  &*(lightgray)  & *(lightgray)  & *(lightgray)  & *(lightgray)  & \;  & *(yellow)  \\
\;  & \;  & \; & \;   & \; & \;  & *(yellow)  \\
\; & \; & \; & \; & \;  & *(yellow)  \\
\; & \; & \; & \;  & *(yellow)  \\
\; & \; & \;  & *(yellow)  \\
\; & \;  & *(yellow)  \\
\; &  *(yellow)  \\
*(yellow) 
\end{ytableau}
\\
m=7,k=0,n=4 \\
m'=8,k'=6,n'=3 \\
}
\hspace{-2mm}
B=
\substack{
\begin{ytableau}
 *(lightgray)  &*(lightgray)  &*(lightgray)  &*(lightgray)  &*(lightgray)  & *(lightgray)  & *(lightgray)  & *(lightgray)  & \;  & *(yellow)  \\
 *(lightgray)  &*(lightgray)  &*(lightgray)  &*(lightgray)  & *(lightgray)  & *(lightgray)  & *(lightgray)  & \;  & *(yellow)  \\
 *(lightgray)  &*(lightgray)  &*(lightgray)  & *(lightgray)  & *(lightgray)  & *(lightgray)  & \;  & *(yellow)  \\
*(lightgray)  & \;  & \; & \;   & \; & \;  & *(yellow)  \\
\; & \; & \; & \; & \;  & *(yellow)  \\
\; & \; & \; & \;  & *(yellow)  \\
\; & \; & \;  & *(yellow)  \\
\; & \;  & *(yellow)  \\
\; &  *(yellow)  \\
*(yellow) 
\end{ytableau}
\\
m=7,k=1,n=3 \\
}
\hspace{-2mm}
C=
\substack{
 \begin{ytableau}
 *(lightgray)  &*(lightgray)  &*(lightgray)  &*(lightgray)  &*(lightgray)  & *(lightgray)  & *(lightgray)  & *(lightgray)  & \;  & *(yellow)  \\
 *(lightgray)  &*(lightgray)  &*(lightgray)  &*(lightgray)  & *(lightgray)  & *(lightgray)  & *(lightgray)  & \;  & *(yellow)  \\
 *(lightgray)  &*(lightgray)  &*(lightgray)  & *(lightgray)  & *(lightgray)  & *(lightgray)  & \;  & *(yellow)  \\
*(lightgray)  &*(lightgray)  & \; & \;  & \; & \;  & *(yellow)  \\
\; & \; & \; & \; & \;  & *(yellow)  \\
\; & \; & \; & \;  & *(yellow)  \\
\; & \; & \;  & *(yellow)  \\
\; & \;  & *(yellow)  \\
\; &  *(yellow)  \\
*(yellow) 
\end{ytableau}
\\
m=7,k=2,n=3
}
\hspace{-2mm}
D=
\substack{
 \begin{ytableau}
 *(lightgray)  &*(lightgray)  &*(lightgray)  &*(lightgray)  &*(lightgray)  & *(lightgray)  & *(lightgray)  & *(lightgray)  & \;  & *(yellow)  \\
 *(lightgray)  &*(lightgray)  &*(lightgray)  &*(lightgray)  & *(lightgray)  & *(lightgray)  & *(lightgray)  & \;  & *(yellow)  \\
 *(lightgray)  &*(lightgray)  &*(lightgray)  & *(lightgray)  & *(lightgray)  & *(lightgray)  & \;  & *(yellow)  \\
*(lightgray)  &*(lightgray)  & *(lightgray)  & \;   & \; & \;  & *(yellow)  \\
\; & \; & \; & \; & \;  & *(yellow)  \\
\; & \; & \; & \;  & *(yellow)  \\
\; & \; & \;  & *(yellow)  \\
\; & \;  & *(yellow)  \\
\; &  *(yellow)  \\
*(yellow) 
\end{ytableau}
\\
m=7,k=3,n=3
}
\hspace{-2mm}
E=
\substack{
\begin{ytableau}
 *(lightgray)  &*(lightgray)  &*(lightgray)  &*(lightgray)  &*(lightgray)  & *(lightgray)  & *(lightgray)  & *(lightgray)  & \;  & *(yellow)  \\
 *(lightgray)  &*(lightgray)  &*(lightgray)  &*(lightgray)  & *(lightgray)  & *(lightgray)  & *(lightgray)  & \;  & *(yellow)  \\
 *(lightgray)  &*(lightgray)  &*(lightgray)  & *(lightgray)  & *(lightgray)  & *(lightgray)  & \;  & *(yellow)  \\
*(lightgray)  &*(lightgray)  & *(lightgray)  & *(lightgray)   & *(lightgray) & *(lightgray)  & *(yellow)  \\
\; & \; & \; & \; & \;  & *(yellow)  \\
\; & \; & \; & \;  & *(yellow)  \\
\; & \; & \;  & *(yellow)  \\
\; & \;  & *(yellow)  \\
\; &  *(yellow)  \\
*(yellow) 
\end{ytableau}
\\
m=7,k=6,n=3
}
\]
We can use the recursion in \cref{prop:lee} repeatedly to express LLT polynomials 
by adding one and removing one shaded box in row $m$.
For example, $C$ can be expressed in terms of $B$ and $D$.
Similarly, $B$ can be expressed in terms of $A$ and $C$,
and we get a system of linear equations expressing $B$, $C$ and $D$
in terms of only $A$ and $E$.

When $k=m-1$ (as for $E$ above) the graph $\Gamma_\avec$ is a disjoint union 
of a complete graph and a line graph,
which is a base case. Furthermore, when $k=0$, (as for $A$ above) we obtain
a melting lollipop graph with the new parameters $m'=m+1$, $k'=m-2$ and $n'=n-1$,
which are dealt with by induction over $n$.

\section{The Hall--Littlewood case}\label{sec:hallLittlwood}

In \cite{HaglundHaimanLoehr2005}, the modified Macdonald polynomials $\macdonaldH_\lambda(\xvec;q,t)$ 
are expressed as a positive sum of certain LLT polynomials.
The modified Macdonald polynomials specialize to modified Hall--Littlewood polynomials at $q=0$,
which in turn are closely related to the transformed Hall--Littlewood polynomials.

\begin{definition}[See \cite{DesarmenienLeclercThibon1994,TudoseZabrocki2003} for a background]\label{def:hallLittlewoodH}
Let $\lambda$ be a partition.
The \defin{transformed Hall--Littlewood polynomials} are defined as
\[
 \hallLittlewoodT_\mu(\xvec;q) = \sum_{\lambda} K_{\lambda\mu}(q)\schurS_\lambda(x)
\]
where $K_{\lambda\mu}(q)$ are the Kostka--Foulkes polynomials.
\end{definition}
The $\hallLittlewoodT_\lambda$ are sometimes denoted $Q'_\lambda$ and is the adjoint basis 
to the Hall--Littlewood $P$ polynomials for the standard Hall scalar product,
see \cite{DesarmenienLeclercThibon1994}.
A more convenient definition of the transformed Hall--Littlewood polynomials is the following.
For $\lambda \vdash n$ we have
\begin{equation}\label{eq:hallLittlewoodHDef}
 \hallLittlewoodT_\lambda(\xvec;q) = \prod_{1\leq i<j \leq n} \frac{1 - R_{ij}}{1 - qR_{ij}} \completeH_{\lambda}(\xvec)
\end{equation}
where $R_{ij}$ are \defin{raising operators} acting on the 
partitions (or compositions) indexing the complete homogeneous symmetric functions
as
\[
 R_{ij} \completeH_{(\lambda_1,\dotsc,\lambda_n)}(\xvec) = 
 \completeH_{(\lambda_1,\dotsc,\lambda_i +1, \dotsc, \lambda_j-1,\dotsc,\lambda_n)}(\xvec).
\]
Note that $\schurS_\lambda(\xvec) = \hallLittlewoodT_\lambda(\xvec;0)$,
and \eqref{eq:hallLittlewoodHDef} gives 
$
 \schurS_\lambda(\xvec) = \prod_{i<j} (1 - R_{ij})\completeH_{\lambda}(\xvec)
$
which is just the Jacobi--Trudi identity for Schur functions in disguise.
Furthermore, note that \eqref{eq:hallLittlewoodHDef} immediately implies that
\begin{equation}\label{eq:HLdominance}
\hallLittlewoodT_\lambda(\xvec;q) =  
\completeH_{\lambda}(\xvec) + \sum_{\mu \vartriangleright \lambda} c_\mu(q) \completeH_{\mu}(\xvec), \qquad c_\mu(q) \in \setZ[q]
\end{equation}
where $\vartriangleright$ denotes the dominance order, since the 
raising operators $R_{ij}$ can only make partitions larger in dominance order.

\medskip 

We now connect the transformed Hall--Littlewood polynomials 
with certain vertical strip LLT polynomials.
\begin{definition}\label{def:HLDiagram}
Given a partition $\mu \vdash n$, 
let $s_i$ be defined as $s_i \coloneqq \mu_1+\dotsb+\mu_i$, with $s_0 \coloneqq 1$.
From $\mu$, we construct a vertical strip diagram $\Gamma_{\mu}$ on $n$
vertices with the following edges: 
\begin{enumerate}[(a)]
 \item for each $j=1,\dotsc,\length(\mu)$, 
 let the vertices
 $\{s_{j-1},\dotsc,s_{j}\}$ 
constitute a complete subgraph of $\Gamma_{\mu}$,
 \item for each $j=2,\dotsc,\length(\mu)$, we also have the $\binom{\mu_j}{2}$ edges
 \[
  \{ s_{j-1}-i \to s_{j-1}+k+1 \text{ whenever } 0\leq i,k \text{ and } i + k \leq \mu_j-1\}.
 \]
\end{enumerate}
Thus, for each $j\geq 2$ in item (b),
we obtain a triangular shape of boxes with edges, as marked in \eqref{eq:smallTri}, where $\mu_j=5$.
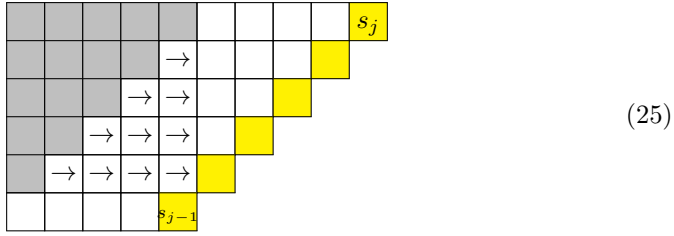
\begin{equation}\label{eq:smallTri}
\ytableausetup{boxsize=1.4em}
\begin{ytableau} 
 *(lightgray) & *(lightgray) &*(lightgray) &*(lightgray) &*(lightgray) & \; &\;&\; &\; & *(yellow) s_j \\
 *(lightgray)  &*(lightgray)  & *(lightgray)  & *(lightgray)  & \rightarrow&  & & & *(yellow)  \\
 *(lightgray)  &*(lightgray)  & *(lightgray)  &  \rightarrow & \rightarrow &  && *(yellow) \\
 *(lightgray)  &*(lightgray)  & \rightarrow &  \rightarrow & \rightarrow &   & *(yellow) \\
 *(lightgray)  &\rightarrow & \rightarrow &  \rightarrow & \rightarrow & *(yellow)\\
 \; & \; & \; &  \; & *(yellow) \scriptstyle{s_{j-1}} & \none\\
\end{ytableau}
\end{equation}

Furthermore, all outer corners are taken as strict edges, see \cref{ex:HLverticalStrips} below.
As before, let $O(\Gamma_\mu)$ denote the set of orientations of the edges of $\Gamma_\mu$.
\end{definition}

\begin{proposition}
Let $\mu$ be a partition and let $\Gamma_\mu$ be the vertical strip diagram
constructed from $\mu$ and let $\LLT_{\mu}(\xvec;q)$ be the corresponding LLT polynomial. 
Then 
\begin{equation}\label{eq:relationWithHL}
 \omega \LLT_{\mu}(\xvec;q)  = q^{\sum_{i\geq 2} \binom{\mu_i}{2} } \hallLittlewoodT_{\mu'}(\xvec;q).
\end{equation}
\end{proposition}
\begin{proof}[Brief proof sketch]
We use \cite[A.59]{qtCatalanBook} which states that for any partition $\lambda$,
the coefficient of $t^{n(\lambda)}$ in the modified Macdonald polynomial $\macdonaldH_{\lambda}(\xvec;q,t)$
is almost a transformed Hall--Littlewood polynomial:
\[
  [t^{n(\lambda)}]\macdonaldH_{\lambda}(\xvec;q,t) = \omega \hallLittlewoodT_{\lambda'}(\xvec;q).
\]
The $\macdonaldH_{\lambda}(\xvec;q,t)$ is a sum over certain LLT polynomials
and in particular, the coefficient of $t^{n(\lambda)}$
is a single vertical-strip LLT polynomial, multiplied with $q^{-A}$, where 
$A$ is the sum of arm lengths in the diagram $\lambda$.
Unraveling the definitions in \cite[A.14]{qtCatalanBook}, 
we arrive at the identity\footnote{It was pointed out by the referee that \eqref{eq:relationWithHL} also follows 
directly from \cite[Thm. 6.8]{qtCatalanBook}.}
in \eqref{eq:relationWithHL}.
\end{proof}

\begin{example}\label{ex:HLverticalStrips}
The Hall--Littlewood polynomial $\hallLittlewoodT_{3321}(\xvec;q)$
is related to the vertical strip diagram $\Gamma_{432}$ in \eqref{eq:relationWithHL}.

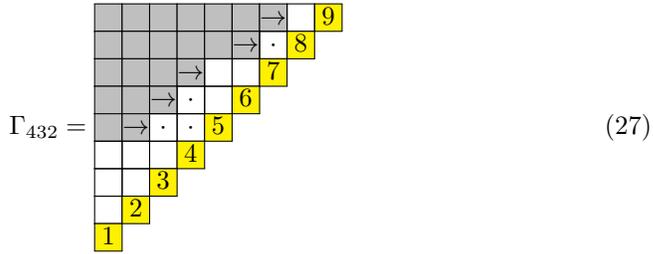
\begin{align}
\ytableausetup{boxsize=1.0em}
\Gamma_{432} =
\begin{ytableau} 
*(lightgray)   & *(lightgray)   & *(lightgray) & *(lightgray) & *(lightgray)  &   *(lightgray)   & *(lightgray) \rightarrow  &  & *(yellow)  9 \\
*(lightgray)   & *(lightgray)   & *(lightgray) & *(lightgray) & *(lightgray)  & *(lightgray) \rightarrow & \cdot & *(yellow) 8  \\
*(lightgray)   & *(lightgray) & *(lightgray) & *(lightgray) \rightarrow  &    &  & *(yellow)  7 \\
*(lightgray) & *(lightgray) & *(lightgray) \rightarrow  &  \cdot  &   & *(yellow)  6 \\
*(lightgray)  &  *(lightgray) \rightarrow  & \cdot  & \cdot & *(yellow)  5 \\
  &   &   & *(yellow)  4 \\
  &   & *(yellow)  3 \\
  & *(yellow) 2 \\
*(yellow)  1 \\
\end{ytableau}
\end{align}
The edges marked with a dot are the edges in item (b). 
There are $\sum_{i\geq 2} \binom{\mu_i}{2}$ such edges. 
Notice that the vertex partition of this orientation is $\{974,863,52,1\}$.
Furthermore, it is fairly straightforward to see 
that for \emph{any} orientation $\theta$ of $\Gamma_{\mu}$, 
we must have that the partition $\hrvpp(\theta)$ dominates $\mu'$.
\end{example}

We can now easily give some strong support for \cref{conj:main}.
\begin{corollary}
For any partition $\mu$,  the vertical-strip LLT polynomial $\LLT_{\mu}(\xvec;q+1)$
is $\elementaryE$-positive.
\end{corollary}
\begin{proof}
 Using \eqref{eq:relationWithHL}, it suffices to prove that $\hallLittlewoodT_{\lambda'}(\xvec;q+1)$
 is $\completeH$-positive.  From \eqref{eq:hallLittlewoodHDef}, we have that 
\begin{align}
\hallLittlewoodT_{\mu'}(\xvec;q+1) &= \prod_{i<j} \frac{1-R_{ij}}{1-(q+1)R_{ij}} \completeH_{\mu'}(\xvec) \\
&= \prod_{i<j}(1-R_{ij})(1+(q+1)R_{ij} + (q+1)^2 R^2_{ij} + \dotsb) \completeH_{\mu'}(\xvec) \\
&=\prod_{i<j}
(1+qR_{ij} + q(q+1) R^2_{ij}+ q(q^2+1) R^3_{ij} + \dotsb) \completeH_{\mu'}(\xvec) \\
&=\prod_{i<j} \left( 1 + \sum_{t \geq 1} q(1+q)^{t-1} R^t_{ij} \right) \completeH_{\mu'}(\xvec) \label{eq:product}.
\end{align}
This proves positivity.
\end{proof}

\begin{problem}
 Find a bijective proof that $\LLTc_\mu(\xvec;q+1)$ is equal to $\LLT_\mu(\xvec;q+1)$,
 by interpreting each term in \cref{eq:product},
 and combine with \eqref{eq:relationWithHL}.
 
It is tempting to believe that summing over the orientations of $\Gamma_\mu$ in \cref{def:HLDiagram}
where all edges in condition (b) are oriented in a non-descending manner 
would give exactly $\omega \hallLittlewoodT_{\mu'}(\xvec;q+1)$. 
However, this fails for $\mu=222$.
\end{problem}

\section{Generalized cocharge and \texorpdfstring{$\elementaryE$}{e}-positivity}\label{sec:cocharge}

In \cite{HuhNamYoo2018}, the authors consider a certain classes of unicellular LLT polynomials 
that can be expressed in a particularly nice way.
These are polynomials indexed by complete graphs, line graphs and a few other families.
In this section, we prove that the corresponding 
LLT polynomials are positive in the elementary basis.
In fact, we do this by giving a rather surprising relationship between 
a type of cocharge and orientations.
\medskip 

For a semi-standard Young tableau $T$, the reading word is formed by reading the boxes of 
$\lambda$ row by row from bottom to top, and from left to right within each row.
The \defin{descent set} of a standard Young tableau $T$ is defined as
\[
 \Des(T) \coloneqq \{ i \in [n-1] : \text{$i+1$ appear before $i$ in the reading word}  \}.
\]
Given a Dyck diagram $\avec$, we define the \defin{weight} as
\begin{align}
 \wt_\avec(T) = \sum_{i \in \Des(T)} a_{n+1-i}.
\end{align}
The weight here is also known as \defin{cocharge} whenever $\avec$ is the complete graph $(0,1,2,\dotsc,n-1)$,
see for example \cite{qtCatalanBook}.
If we let $T'$ denote the transposed tableau, then for any $T$ and $\avec$, we have
\[
\Des(T') = [n-1]\setminus \Des(T) \text{ and }
 \wt_\avec(T') = (a_1+\dotsb+a_n) - \wt_\avec(T).
\]
It will be convenient to define 
\begin{align}\label{eq:chargeDef}
\modwt_\avec(T)\coloneqq \wt_\avec(T') = \sum_{i \notin \Des(T)} a_{n+1-i}. 
\end{align}

\begin{example}
Let $\avec=(0,1,2,3,3,2,2,3)$ and 
\[
 T = 
 \begin{ytableau}
  1 & 3 & 4 \\
  2 & 6 & 8 \\
  5 \\
  7
 \end{ytableau}
\]
The reading word of $T$ is $75268134$, $\Des(T) = \{1,4,6\}$ 
so $\wt_\avec(T) = a_8 + a_5+a_3 = 7$ and $\modwt_\avec(T) = 9$.
\end{example}

\begin{definition}
Given an area sequence $\avec$ of length $n$, we define 
the polynomial
\begin{equation}
 \FLLT_\avec(\xvec;q) \coloneqq \sum_{\lambda \vdash n} \sum_{T \in \SYT(\lambda)} q^{\wt_\avec(T)} \schurS_\lambda(\xvec).
\end{equation}
\end{definition}
From this definition, it follows that
\begin{equation}
 \omega \FLLT_\avec(\xvec;q) = 
 \sum_{\lambda \vdash n} \sum_{T \in \SYT(\lambda)} q^{\modwt_\avec(T)} \schurS_{\lambda}(\xvec).
\end{equation}

The following proposition is a collection of results in \cite{HuhNamYoo2018}.
\begin{proposition}\label{prop:lollipop}
We have that $\FLLT_\avec(\xvec;q) = \LLT_\avec(\xvec;q)$ for the 
the families of graphs listed in \cref{subsec:graphs}: the complete graphs, line graphs, lollipop graphs, 
melting complete graphs and melting lollipop graphs.
\end{proposition}

Given a composition $\gamma$, let
\[
D(\gamma) \coloneqq \{\gamma_1,\gamma_1+\gamma_2,\dotsc,\gamma_1+\gamma_2+\dotsc+\gamma_{\ell}\}.
\]
\begin{lemma}\label{lem:stdBijection}
Let $\lambda\vdash n$ and let $\gamma$ be a composition of $n$ with $\ell$ parts.
Then the standardization map
\[
\std: \{ S \in \SSYT(\lambda,\gamma) \} \to \{ T \in \SYT(\lambda) : \Des(T) \subseteq D(\gamma) \}
\]
is a bijection.
\end{lemma}
\begin{proof}
This is straightforward from the definition of standardization and descents,
see for example \cite[p. 5]{qtCatalanBook}.
\end{proof}

We shall now introduce a different statistic on orientations.
Given $\theta \in O(\Gamma_\avec)$, we say that a vertex $v$ is a \defin{bottom} of $\theta$
if there is no $u<v$ such that $(u,v)$ is ascending in $\theta$.
Let $s_1,\dotsc,s_k$ be the bottoms ordered decreasingly and let $s_0 \coloneqq n+1$.
By definition, vertex $1$ is always a bottom.
Let $\sigma(\theta)$ be defined as the composition of $n$
with the parts given by $\{ s_{i-1}-s_{i} : i=1,\dotsc,k \}$
and note that $D(\sigma(\theta)) = \{n+1-s_i : i=1,2,\dotsc,k-1\}$.

\begin{example}
The orientation $\theta$ in \eqref{eq:sourcesExample} has vertices $1$, $3$ and $6$ as bottoms.
Furthermore, $\sigma(\theta)=(1,3,2)$ and $D(\sigma(\theta)) = \{1,4\}$.
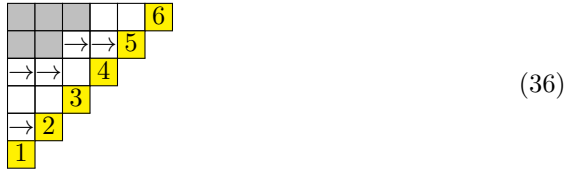
\begin{align}\label{eq:sourcesExample}
\begin{ytableau}
*(lightgray)   & *(lightgray)   & *(lightgray)   &      &    & *(yellow) 6 \\
*(lightgray)   & *(lightgray)   &   \to  & \to    & *(yellow) 5\\
  \to  &  \to  &     & *(yellow) 4\\
    &     & *(yellow) 3\\
  \to  & *(yellow) 2\\
*(yellow) 1
\end{ytableau}
\end{align}
Note that $\hrvpp(\theta)=(5,1)$ so $\sigma$ and $\hrvpp$ are indeed very different.
\end{example}

The following theorem was proved for the complete graph and the line graph in
\cite{AlexanderssonPanova2016}. We can now generalize it to 
all unit interval graphs.
\begin{theorem}
Let $\avec$ be an area sequence of length $n$.
Then
\begin{align}\label{eq:chargeAsOrientations}
\FLLT_\avec(\xvec;q+1)
=
\sum_{\theta \in O(\Gamma_\avec)} q^{\asc(\theta)} \elementaryE_{\sigma(\theta)}(\xvec).
\end{align}
\end{theorem}
\begin{proof}
We apply $\omega$ on both sides of \cref{eq:chargeAsOrientations}, so it suffices to prove that
\begin{align}\label{eq:chargeAsOrientations2}
\omega \FLLT_\avec(\xvec;q+1)
=
\sum_{\theta \in O(\Gamma_\avec)} q^{\asc(\theta)} \completeH_{\sigma(\theta)}(\xvec).
\end{align}
Recall, in e.g. \cite{Macdonald1995}, the standard expansion
\begin{align}
 \completeH_\nu(\xvec)= \sum_{\lambda} K_{\lambda,\nu} \schurS_{\lambda}(\xvec),
\end{align}
where $K_{\lambda,\nu} = |\SSYT(\lambda,\nu)|$ are the Kostka coefficients. 
Thus, comparing both sides of \eqref{eq:chargeAsOrientations2} in the Schur basis,
it suffices to show that for every partition $\lambda$,
\begin{align*}
\sum_{T \in \SYT(\lambda)} (1+q)^{\modwt_\avec(T)}
=
\sum_{\theta \in O(\Gamma_\avec)} q^{\asc(\theta)} K_{\lambda,\sigma(\theta)}.
\end{align*}
Using \cref{lem:stdBijection} in the right hand side and unraveling the definition in the left hand side, 
it is enough to prove that
\begin{align*}
\sum_{T \in \SYT(\lambda)} 
\prod_{i \notin \Des(T)} (1+q)^{a_{n+1-i}}
=
\sum_{T \in \SYT(\lambda)}
\sum_{\substack{\theta \in O(\Gamma_\avec) \\ \Des(T) \subseteq D(\sigma(\theta)) }} q^{\asc(\theta)}.
\end{align*}
It then suffices to prove that for a fixed $T \in \SYT(\lambda)$ we have
\begin{align}\label{eq:sytBijEq2}
\prod_{i \notin \Des(T)} (1+q)^{a_{n+1-i}}
=
\sum_{\substack{\theta \in O(\Gamma_\avec) \\ \Des(T) \subseteq D(\sigma(\theta)) }} q^{\asc(\theta)}.
\end{align}
Both sides may now be interpreted as a weighted sum over all
orientations of $\Gamma_\avec$ where no ascending edges end in $\{i : n+1-i \in \Des(T) \}$.
\end{proof}

\begin{corollary}
All families of unicellular LLT polynomials $\LLT_\avec(\xvec;q+1)$
indexed by complete graphs, line graphs, lollipop graphs and melting lollipop graphs are $\elementaryE$-positive.
\end{corollary}

Notice that the formula in \eqref{eq:chargeAsOrientations} is different from
the conjectured formula in \cref{conj:main}, since $\hrvpp(\theta)$ and $\sigma(\theta)$
are different. 
This is not surprising as $\LLT_\avec(\xvec;q)$ and $\FLLT_\avec(\xvec;q)$
are not equal for general $\avec$. 
However, it is rather remarkable that \cref{conj:main} implies 
that \eqref{eq:chargeAsOrientations} and \cref{eq:conjFormula} agree whenever 
$\LLT_\avec(\xvec;q) = \FLLT_\avec(\xvec;q)$.

\section{A possible approach to settle the main conjecture}\label{sec:powersum}

In \cite{AlexanderssonPanova2016} and later in \cite{AlexanderssonSulzgruber2018} (with a different approach) 
we gave formulas for the power-sum expansion of all vertical-strip LLT polynomials. 
The unicellular case is a straightforward consequence of \cref{lem:plethRelation} (see \cite{AlexanderssonPanova2016,HaglundWilson2017})
together with the power-sum expansion formula for the chromatic symmetric symmetric functions.
We note that the formula in the chromatic case was first conjectured by Shareshian--Wachs and later proved by Athanasiadis \cite{Athanasiadis15}.

It is straightforward to expand \eqref{eq:conjFormula} in the power-sum basis,
so to settle \cref{conj:main}, it suffices to show that 
$\omega \LLT_\avec(\xvec;q+1) = \omega \LLTc_\avec(\xvec;q+1)$ for all $\avec$ by comparing 
coefficients of $\psumP_\lambda/z_\lambda$.
We shall now introduce the necessary terminology from \cite{AlexanderssonSulzgruber2018} 
to state  \cref{conj:main} in this form.
\medskip 

For any subset $S \subseteq E(\Gamma_\avec)$,
let $P(S)$ denote the poset given by the transitive closure of the edges in $S$.
Given a poset $P$ on $n$ elements, let $\opsurj(P)$ be the set of order-preserving 
surjections $f: P\to [k]$ for some $k$.
The \defin{type} of a surjection $f$ is defined as
\[
 \alpha(f) \coloneqq (|f^{-1}(1)|,|f^{-1}(2)|,\dotsc,|f^{-1}(k)|),
\]
and this is a composition of $n$ with $k$ parts.
Let $\opsurj_\alpha(P) \subseteq \opsurj(P)$ be the set of surjections of type $\alpha$.
Finally, let $\opsurj^\ast_\alpha(P) \subseteq \opsurj_\alpha(P)$
be the set of surjections $f \in \opsurj_\alpha(P)$ such that for each $j \in [k]$, 
$f^{-1}(j)$ is a subposet of $P$ with a unique minimal element.

\begin{proposition}[{See \cite[Thm. 5.6, Thm. 7.10]{AlexanderssonSulzgruber2018}}]\label{prop:pexpansion}
The power-sum expansion of $\omega \LLT_{\avec}(\xvec;q+1)$ is given as
 \begin{equation}
 \omega \LLT_{\avec}(\xvec;q+1) = \sum_{\theta \in O(\avec)} q^{\asc(\theta)} 
 \sum_{\lambda \vdash n} \frac{\psumP_\lambda(\xvec)}{z_\lambda} |\opsurj^\ast_\lambda(P(\theta))|
\end{equation}
where $P(\theta)$ is the poset on $[n]$ and edges given by the 
transitive closure of the ascending edges in $\theta$.
\end{proposition}

The family of functions $\LLTc_{\avec}(\xvec;q+1)$ has a similar 
expansion in terms of the power-sum symmetric functions.
\begin{lemma}
 The power-sum expansion of $\omega \LLTc_{\avec}(\xvec;q+1)$ is given as
 \begin{equation}
 \omega \LLTc_{\avec}(\xvec;q+1) = \sum_{\theta \in O(\avec)} q^{\asc(\theta)} 
 \sum_{\lambda \vdash n} \frac{\psumP_\lambda(\xvec)}{z_\lambda} |\opsurj^\ast_\lambda(B(\theta))|
\end{equation}
where $B(\theta)$ is the poset consisting of a disjoint union of chains with lengths
given by $\hrvpp(\theta)$.
\end{lemma}
\begin{proof}
 This follows easily from the definition of $\LLTc_{\avec}(\xvec;q+1)$
 and the expansion of the elementary symmetric functions into power-sum symmetric functions,
 see \cite{Egecioglu1991} and \cite[Section 7]{AlexanderssonSulzgruber2018}.
\end{proof}

\begin{conjecture}[Equivalent with \cref{conj:main}]
For any area sequence $\avec$ of length $n$ and partition $\lambda \vdash n$,
\[
 \sum_{\theta \in O(\avec)} q^{\asc(\theta)} |\opsurj^\ast_\lambda(P(\theta))|
 =\sum_{\theta \in O(\avec)} q^{\asc(\theta)} |\opsurj^\ast_\lambda(B(\theta))|.
\]
\end{conjecture}

\subsection*{Acknowledgements}

The author would like to thank Qiaoli Alexandersson, 
Fran{\c{c}}ois Bergeron, Svante Linusson, 
Greta Panova and Robin Sulzgruber for
helpful discussions and suggestions. 
The author also appreciates the comments by the anonymous referees which improved the paper.
This research has been funded by the \emph{Knut and Alice Wallenberg Foundation} (2013.03.07).

\bibliographystyle{alphaurl}
\bibliography{bibliography}

\end{document}